\documentclass[11pt]{amsart}
\usepackage{graphicx}
\usepackage{hyperref}
\usepackage{setspace}
\usepackage{enumerate}
\usepackage{amsmath,amsfonts,amssymb,amscd}

\newtheorem{thm}{Theorem}[section]
\newtheorem{cor}[thm]{Corollary}
\newtheorem{lem}[thm]{Lemma}
\newtheorem{prop}[thm]{Proposition}

\theoremstyle{definition}
\newtheorem{defn}[thm]{Definition}

\theoremstyle{remark}

\numberwithin{equation}{section}

\begin{document}

\title[]{Zeroes of polynomials with prime inputs \\  and Schmidt's $h$-invariant}

\author{Stanley Yao Xiao}
\address{ Mathematical Institute\\ University of Oxford \\
Oxford, UK }
\email{Stanley.Xiao@maths.ox.ac.uk}

\author{Shuntaro Yamagishi}
\address{Department of Mathematics \& Statistics \\
Queen's University \\
Kingston, ON\\  K7L 3N6 \\
Canada}
\email{sy46@queensu.ca}
\indent

\date{Revised on \today}

\begin{abstract}
In this paper we show that a polynomial equation admits infinitely many prime-tuple solutions assuming only that the equation satisfies suitable local conditions and the polynomial is sufficiently non-degenerate algebraically.
Our notion of algebraic non-degeneracy is related to the $h$-invariant introduced by W. M. Schmidt. Our results prove a conjecture of B. Cook and \'{A}. Magyar \cite{CM} for hypersurfaces of degrees $2$ and $3$.
\end{abstract}



\maketitle

\section{Introduction}

Solving systems of integral polynomial equations in integers is among the oldest and persistently interesting problems in number theory. It is understood, especially in the context of the Hardy-Littlewood circle method, that systems tend to become easier to solve when the number of variables involved increases. For instance, it is not known whether the equation $x^2 + 1 = p$ where $x$ varies in the integers and $p$ varies among the primes has infinitely many solutions, but the corresponding 3-variable equation $x^2 + y^2 = p$ was solved by Fermat using elementary means over three centuries ago.
One can then ask whether
it is possible to interpolate between these situations. That is, given a system of polynomial equations which is solvable in the integers, one can ask whether the system remains solvable when some of the variables are restricted to a thin subset of integers.
One particular natural subset is the set of prime numbers. Indeed, many interesting problems involving prime numbers may be phrased in such a manner. For example, the existence of infinitely many solutions to the equation $x - y = 2$
with $x,y$ restricted to primes is precisely the twin prime conjecture.


In \cite{CM}, B. Cook and \'{A}. Magyar broke new ground by applying the Hardy-Littlewood circle method to show, in great generality, that systems of polynomial equations in many variables can be solved when all of the inputs are prime numbers. The key hypothesis they require is that the so-called Birch singular locus must be sufficiently small. For $\mathbf{f} = \{f_1, \ldots , f_{r_d}  \} \subseteq \mathbb{Q}[x_1, \ldots , x_n]$ a system of forms (homogeneous polynomials) of degree $d$, we define the \emph{Birch singular locus} $V_{\mathbf{f}}^*$ to be the affine variety in $\mathbb{A}_{\mathbb{C}}^n$ given by
$$
V_{\mathbf{f}}^* = \Big\{ \mathbf{x} \in \mathbb{C}^n : \text{rank } \left( \frac{ \partial f_r(\mathbf{x}) }{ \partial x_j }\right)_{ \substack{ 1 \leq r \leq r_d  \\ 1 \leq j \leq n } }  < r_d \Big\}
,
$$
and let the \textit{Birch rank} to be $\mathcal{B} (\mathbf{f}) = n - \dim V_{\mathbf{f}}^*$.
The Birch rank is an important invariant that arose in \cite{B}. In \cite{S}, W. M. Schmidt introduced a different invariant, now called Schmidt's $h$-invariant, for systems of polynomials.  B. Cook and \'{A}. Magyar conjectured in \cite[pp. 736]{CM} that their main theorem ought to hold assuming the largeness of the $h$-invariant
instead of the Birch rank (see (\ref{relnbtwbandh})).

In this paper, we give a partial solution to the conjecture of B. Cook and \'{A}. Magyar. We establish the conjecture for hypersurfaces with an additional assumption. However, our assumption is redundant for quadratic polynomials and cubic polynomials; therefore, we establish the conjecture for quadratic and cubic polynomials.
Given a form $f \in \mathbb{Q}[x_1, \ldots, x_n]$ of degree at least $2$, we define the \emph{$h$-invariant}
$h(f)$ of $f$ to be the least positive integer $h$ such that $f$ can be written identically as
\begin{equation}
\label{defn of hinv}
f = U_1 V_1 + \cdots + U_h V_h,
\end{equation}
where each $U_i$ and $V_i$ are forms in $\mathbb{Q}[x_1, \ldots, x_n]$ of degree at least $1 \ (1 \leq i \leq h)$.
We then define the following quantity
$$
h^{\star}(f) = \max ( |\{ U_i  : \deg U_i = 1 \}| ),
$$
where the maximum is over all representations of the shape (\ref{defn of hinv}).
In other words, $h^{\star}(f)$ is the maximum number of linear forms involved
in the representation of $f$ as a sum of $h = h(f)$ products of rational forms.
Clearly, we have $h^{\star}(f) \leq h(f).$
For a degree $d$ polynomial $b(\mathbf{x}) \in \mathbb{Q}[x_1, \ldots, x_n]$, we define $h(b) = h(f)$
where $f(\mathbf{x})$ is the degree $d$ portion of $b(\mathbf{x})$.
We note that any polynomial $b(\mathbf{x})$ of degree $2$ or degree $3$ satisfies
$$
h(b) = h^{\star}(b).
$$

We define the following quantity
$$
\mathcal{M}_{b}(N) = \sum_{\mathbf{x} \in [0, N]^n \cap \mathbb{Z}^n} \delta_{b }(\mathbf{x} ),
$$
where
\begin{eqnarray}
\notag
\delta_{b} (\mathbf{x} )
=
\left\{
    \begin{array}{ll}
         \prod_{1 \leq i \leq n} \log p_i
         & \ \mbox{if } x_i = p_i^{t_i}, p_i \mbox{ is prime}, t_i \in \mathbb{N} \ (1 \leq i \leq n) \
          \mbox{ and } b(\mathbf{x}) = 0 ,
         \\
         & \\
         0
         & \ \mbox{otherwise }.
    \end{array}
\right.
\notag
\end{eqnarray}

Let $\Lambda$ be the von Mangoldt function, where
$\Lambda(x)$ is $\log p$ if $x$ is a power of a prime $p$ and $0$ otherwise.
We use the notation $e(x)$ to denote $e^{2\pi i x}$. We define
\begin{equation}
\label{def of exp sum T}
T(b; \alpha) = \sum_{\mathbf{x} \in [0, N]^n \cap \mathbb{Z}^n} \Lambda(\mathbf{x}) \ e( \alpha \cdot b(\mathbf{x})),
\end{equation}
where
$$
\Lambda(\mathbf{x}) = \Lambda(x_1) \cdots \Lambda(x_n)
$$
for $\mathbf{x} = (x_1, \ldots , x_n) \in (\mathbb{Z}_{\geq 0})^n$.
By the orthogonality relation, we have
\begin{equation}
\label{defn of MbN}
\mathcal{M}_{b}(N) = \sum_{\mathbf{x} \in [0, N]^n \cap \mathbb{Z}^n} \delta_b(\mathbf{x} ) = \int_0^1 T(b; \alpha) \ d\alpha. 
\end{equation}
We obtain the following theorem by estimating the integral in ~(\ref{defn of MbN}).
\begin{thm}
\label{the main theorem}
Let $b(\mathbf{x}) \in \mathbb{Z}[x_1, \ldots, x_n]$ be a polynomial of degree $d$.
Then there exists a positive number $A_{d}$ dependent only on $d$ such that the following holds.
If $h^{\star}(b) > A_d$, then there exist $c>0$ and $C_b$ such that
$$
\mathcal{M}_{b}(N) =  
C_b N^{n-d} + O\left( \frac{N^{n-d}}{(\log N)^{c} } \right).
$$
\end{thm}
In fact, we prove that $C_b > 0$ provided the equation $b(\mathbf{x}) =0$ has a non-singular solution in $\mathbb{Z}_p^{\times}$,
the units of $p$-adic integers, for every prime $p$ and the equation $f(\mathbf{x}) = 0$, where $f(\mathbf{x})$ is the degree $d$ portion of $b(\mathbf{x})$, has a non-singular real zero in the interior of $\mathfrak{B}_0 = [0,1]^n$.

The following result is an immediate consequence of Theorem \ref{the main theorem}, which replaces the assumption of large Birch rank
in \cite[Theorem 1]{CM} with large $h$-invariant for quadratic and cubic polynomials.
\begin{cor}
\label{the main theorem 2}
Let $d = 2$ or $3$, and $b(\mathbf{x}) \in \mathbb{Z}[x_1, \ldots, x_n]$ be a polynomial of degree $d$.
Then there exists a positive number $A_{d}$ dependent only on $d$ such that the following holds.
If $h(b) > A_d$, then there exist $c>0$ and $C_b$ such that
$$
\mathcal{M}_{b}(N) = C_b N^{n-d} + O\left( \frac{N^{n-d}}{(\log N)^c } \right).
$$
\end{cor}
We establish Theorem \ref{the main theorem 2} in a similar manner to \cite{CM}, but we shall make use of the fact that the representation
~(\ref{defn of hinv}) has enough linear terms. We also modify the method in \cite{CM} to better suit our purposes, so that it
is in terms of the $h$-invariant instead of the Birch rank.

Despite Theorem \ref{the main theorem} and Corollary \ref{the main theorem 2} being our primary goals in this paper, it is necessary for us to work over a system of polynomials at times. Indeed, our strategy is to decompose a polynomial into a sum of elements in a suitable system of polynomials, and then use methods which apply to systems to deduce results of a single polynomial. 

The organization of the rest of the paper is as follows. In Section \ref{section h-inv}, we prove some basic
properties of the $h$-invariant. 
A sufficiently large $h^{\star}(b)$ allows us to massage our polynomial $b(\mathbf{x})$ into something amenable to the circle method,
through a process called the `regularization'. We collect results related to the regularization process in Section \ref{Sec reg lem}.
In Section \ref{technical estimate}, we obtain results from \cite{S} based on Weyl differencing in terms of polynomials instead of forms as in \cite{S}. We chose to present the details in Section \ref{technical estimate} to make certain dependency
of the constants explicit, because it plays an important role in our estimates. We then obtain the minor arc estimates in
Section \ref{section minor arc}, and the major arc estimates in Section \ref{section major arcs}.

\subsection*{Acknowledgments} We would like to thank D. Schindler and the anonymous referees for many helpful comments.
We would also like to thank the department of Pure Mathematics at University of Waterloo for their support
as portions of this work were completed while both of the authors were there as graduate students.

\section{Properties of the $h$-invariant}
\label{section h-inv}
Let $\mathbf{f} = \{f_1, \ldots, f_{r_d}  \} \subseteq \mathbb{Q}[x_1, \ldots, x_n]$ be a system of forms of degree $d > 1$. We generalize the
definition of $h$-invariant for a single form, and
define the $h$-invariant of $\mathbf{f}$ by
\begin{equation}
h(\mathbf{f}) = \min_{\boldsymbol{\mu} \in \mathbb{Q}^{r_d} \backslash \{ \boldsymbol{0} \}}  h( \mu_{1} f_1 + \cdots + \mu_{r_d} f_{r_d} ).
\end{equation}
Given an invertible linear transformation $T \in GL_n(\mathbb{Q})$, let $\mathbf{f} \circ T = \{f_1 \circ T, \ldots, f_{r_d} \circ T \}$.
It follows from the definition of the $h$-invariant that
$$
h(\mathbf{f}) = h(\mathbf{f} \circ T).
$$
Let $\mathbf{b} = ( b_1, \ldots, b_{r_d} ) \subseteq \mathbb{Q}[x_1,\ldots, x_n]$ be a system of degree $d$ polynomials.
We let $f_{r}$ to be the degree $d$ portion of $b_r \ (1 \leq r \leq r_d)$, and define
\begin{equation}
h(\mathbf{b}) = h(\{ f_{r} : 1 \leq r \leq r_d  \}).
\end{equation}
It is known that large Birch rank implies large $h$-invariant, since we have
\begin{eqnarray}
\label{relnbtwbandh}
h(\mathbf{f})\geq 2^{1-d} \mathcal{B}(\mathbf{f})
\end{eqnarray}
 by \cite[Lemma 16.1, (10.3), 
(17.1)]{S}.

We prove two basic lemmas regrading the properties of the $h$-invariant in this section.
Let $f \in \mathbb{Q}[x_1, \ldots, x_n]$ be a form of degree $d$. For $1 \leq i \leq n$, let $f |_{x_i=0} = f(x_1, \ldots, x_{i-1}, 0, x_{i+1},\ldots, x_n ) \in \mathbb{Q}[x_1, \ldots, x_n]$, which is either identically $0$ or a form of degree $d$.
Let $h(f) = 0$ if $f$ is identically $0$. We prove the following simple lemma.
\begin{lem}
\label{h ineq 1}
Let ${f} \in \mathbb{Q}[x_1, \ldots, x_n]$ be a form of degree $d>1$.
Then for any  $1 \leq i \leq n$, we have
$$
h( {f} ) -1   \leq h( {f} |_{x_i=0} ) \leq h( {f} ).
$$
\end{lem}

\begin{proof}
Without loss of generality, we consider the case $i=1$. Let us write
\begin{equation}
\label{lem 1 eqn 1}
f(x_1, \ldots, x_n) = x_1 g(x_1, \ldots, x_n) + f(0, x_2, \ldots, x_n).
\end{equation}
Clearly, $g(\mathbf{x})$ is either identically $0$ or a form of degree $d-1$.
Let $h = h( {f} )$ and $h' = h( {f} |_{x_1=0} )$.
By the definition of $h$-invariant, we can find rational forms $U_{j'}, V_{j'} \ (1 \leq j' \leq h')$ of positive degree
that satisfy
$$
f(0, x_2, \ldots, x_n)= U_1 V_1 + \cdots + U_{h'} V_{h'}.
$$
Note if $h'=0$, we assume the right hand side to be identically $0$.
By substituting the above equation into ~(\ref{lem 1 eqn 1}), we obtain
$$
f  = x_1 g+ U_1 V_1 + \cdots + U_{h'} V_{h'}.
$$
Because $g(\mathbf{x})$ is either identically $0$ or a form of degree $d-1$, it follows that
$$
h \leq 1 + h'.
$$

For the other inequality, let $u_{j}, v_{j} \ (1 \leq j \leq h)$ be rational forms of positive degree that
satisfy
\begin{equation}
\label{f=u1v1...}
f  = u_1 v_1 + \cdots + u_{h} v_{h}.
\end{equation}
By substituting $x_1 = 0$ into each form on both sides of the equation, it is clear that
we obtain $h' \leq h$. This completes the proof of the lemma.
We add a remark that in the special case when $f$ satisfies
$$
f  = x_1 v_1 + u_2 v_2 + \cdots + u_{h} v_{h},
$$
in other words when we have $u_1 = x_1$ in ~(\ref{f=u1v1...}),
we easily obtain $h' = h-1$.
\end{proof}
The following is an immediate consequence of Lemma \ref{h ineq 1}.
\begin{lem}
\label{h ineq 1'}
Let $\mathbf{f} = \{f_1, \ldots, f_{r_d} \} \subseteq \mathbb{Q}[x_1, \ldots, x_n]$ be a system of forms of degree $d>1$.
Suppose $h(\mathbf{f})>1$. Then for any  $1 \leq i \leq n$, we have
$$
h( \mathbf{f} ) -1   \leq h( \mathbf{f} |_{x_i=0} ) \leq h( \mathbf{f} ),
$$
where $\mathbf{f} |_{x_i=0} = \{ f_1|_{x_i=0}, \ldots , f_{r_d}|_{x_i = 0} \}$.
\end{lem}

Let $f(\mathbf{x}) \in \mathbb{Q}[x_1, \ldots, x_n]$ be a form, and let $h = h(f)$ and $0 < M  \leq h$. Suppose we have
$$
f = u_1 V_1 + \cdots + u_{M} V_{M} + U_{M+1} V_{M+1} + \cdots + U_h V_h,
$$
where each $u_i$ is a linear rational form $(1 \leq i \leq M)$, and each $U_{i'}$ and $V_j$
are rational forms of positive degree $(M+1 \leq i' \leq h, 1 \leq j \leq h)$. It can be easily
verified that the linear forms $u_1, \ldots, u_M$ are linearly independent over $\mathbb{Q}$.
Then by considering the reduced row echelon form of the matrix formed by the coefficients of $u_1, \ldots, u_M$, and relabeling the
variables if necessary, we may suppose without loss of generality that
\begin{equation}
\label{h-inv decomp after linear transfn}
f = (x_1 + \ell_1) v_1 + \cdots + (x_M + \ell_M) v_M + u_{M+1} v_{M+1} + \cdots + u_h v_h,
\end{equation}
where each $\ell_i$ is a linear form in $\mathbb{Q}[x_{M+1}, \ldots, x_n] \ (1 \leq i \leq M)$, and each $u_{i'}$ and $v_j$
are rational forms of positive degree $(M+1 \leq i' \leq h, 1 \leq j \leq h)$.
We then define $g_M \in \mathbb{Q}[x_1, \ldots, x_n]$ in the following manner,
\begin{eqnarray}
\label{def gM}
f ( x_1, x_2, \ldots, x_n ) 
= g_M (x_1, \ldots, x_n) + f(-\ell_1, \ldots,  - \ell_M, x_{M+1}, \ldots, x_n).
\end{eqnarray}
We note that there is no ambiguity for defining the polynomial
$$
f(-\ell_1, \ldots, -\ell_M, x_{M+1}, \ldots, x_n) \in \mathbb{Q}[x_{M+1}, \ldots, x_n]
$$
obtained by substitution, because each $\ell_i \in  \mathbb{Q}[x_{M+1}, \ldots, x_n] \ (1 \leq i \leq M)$.
It is also clear that
\begin{equation}
\label{gM is 0}
g(-\ell_1, \ldots,  - \ell_M, x_{M+1}, \ldots, x_n) = 0.
\end{equation}

\begin{lem}
\label{h ineq2}
Let $1 \leq M \leq h$. Suppose a degree $d$ form $f(\mathbf{x}) \in \mathbb{Q}[x_1, \ldots, x_n]$ satisfies
~(\ref{h-inv decomp after linear transfn}). Define $g_M(\mathbf{x})$ as in ~(\ref{def gM}). Then we have
$$
h(g_M) \geq M   \ \ \text{  and } \ \  h( f( -\ell_1, - \ell_2, \ldots , - \ell_M, x_{M+1}, \ldots, x_n ) ) = h-M.
$$
\end{lem}

\begin{proof}
Since the linear forms $(x_1 - \ell_1), \ldots , (x_M - \ell_M)$ are linearly independent over $\mathbb{Q}$,
we can find $A \in GL_n(\mathbb{Q})$ such that
$$
\left( \begin{array}{c}
x_1 - \ell_1 \\
\vdots \\
x_M - \ell_M \\
x_{M+1} \\
\vdots \\
x_n \end{array} \right)
=
A \circ \left( \begin{array}{c}
x_1 \\
\vdots \\
x_n \end{array} \right).
$$
Let
$$
\widetilde{f} (\mathbf{x} )= f(A^{} \circ \mathbf{x} ).
$$
We then have
$$
\widetilde{f} (A^{-1} \circ \mathbf{x} )= f( \mathbf{x} ),
$$
and also that $h(\widetilde{f} ) = h(\widetilde{f} \circ A^{-1}  ) =  h(f) = h$.
Because $f(\mathbf{x})$ satisfies ~(\ref{h-inv decomp after linear transfn}),
it follows that $\widetilde{f} (\mathbf{x} )$ satisfies
$$
\widetilde{f} =  x_1 V_1 + \cdots + x_M V_M + U_{M+1} V_{M+1} + \cdots +  U_h V_h,
$$
where each $U_i$ and $V_j$ are rational forms of positive degree $(M+1 \leq i \leq h, 1 \leq j \leq h)$.

Recall each $\ell_i$ is a linear form in $\mathbb{Q}[x_{M+1}, \ldots, x_n] \ (1 \leq i \leq M)$.
Clearly, we have
\begin{eqnarray}
\widetilde{f} (0, \ldots,0, x_{M+1}, \ldots,  x_n ) &=&
f(A^{} \circ (0, \ldots,0, x_{M+1}, \ldots,  x_n ) )
\notag
\\
&=& f( -\ell_1, - \ell_2, \ldots , - \ell_M, x_{M+1}, \ldots, x_n ).
\notag
\end{eqnarray}
Then we can deduce from Lemma \ref{h ineq 1} (see the remark at the end of the proof of Lemma \ref{h ineq 1}) that
$$
h( f( -\ell_1, - \ell_2, \ldots , - \ell_M, x_{M+1}, \ldots, x_n ) )
= h(\widetilde{f} (0, \ldots,0, x_{M+1}, \ldots,  x_n )) = h-M.
$$
It then follows easily from the fact that $h(f) = h$, the definition of $h$-invariant, and ~(\ref{def gM}), that
$$
h(g_M) \geq M,
$$
for otherwise we obtain a contradiction.
\end{proof}

\section{Regularization lemmas}
\label{Sec reg lem}
In this section, we collect results from \cite{CM} and \cite{S} related to regular systems (see Definition \ref{def regular}) and the
regularization process (Proposition \ref{prop reg par}), which played an important role in \cite{CM} to obtain the minor arc estimate.
Throughout this section we use the following notation. Let $d, n >1$, and let $\textbf{f}$ be a system of forms in $\mathbb{Q}[x_1, \ldots, x_n]$ of degree less than or equal to $d$.
We denote $\textbf{f} = ({\textbf{f}}^{(d)}, \ldots, {\textbf{f}}^{(1)})$, where ${\textbf{f}}^{(i)}$
is the subsystem of all forms of degree $i$ in $\textbf{f} \  (1 \leq i \leq d)$. We label the elements of ${\textbf{f}}^{(i)}$ by
$$
{\mathbf{f}}^{(i)} = \{ f^{(i)}_1, \ldots, f^{(i)}_{r_i} \},
$$
where $r_i = | {\mathbf{f}}^{(i)} |$, the number of elements in ${\mathbf{f}}^{(i)}$.

We shall call a system of polynomials regular if it has at most the expected number of integer solutions, which we define formally below.
\begin{defn}
\label{def regular}
Let $d>1$. Let $\boldsymbol{\psi} = (\boldsymbol{\psi}^{(d)}, \ldots, \boldsymbol{\psi}^{(1)})$ be a system of polynomials in $\mathbb{Q}[x_1, \ldots, x_n]$,
where $\boldsymbol{\psi}^{(i)}$ is the subsystem of all polynomials of degree $i$ in
$\boldsymbol{\psi} \  (1 \leq i \leq d)$.
We denote $V_{\boldsymbol{\psi}, \mathbf{0}} (\mathbb{Z})$ to be the set of solutions in $\mathbb{Z}^n$ of the equations
$$
\psi^{(i)}_j (\mathbf{x}) = 0 \ (1 \leq i \leq d, 1 \leq j \leq | \boldsymbol{\psi}^{(i)} |),
$$
which we denote by
$\boldsymbol{\psi}(\mathbf{x}) = \mathbf{0}$.
Let $r_i = | \boldsymbol{\psi}^{(i)} | \ (1 \leq i \leq d)$,
and let $D_{\boldsymbol{\psi}} = \sum_{i=1}^d i r_i$. We say the system $\boldsymbol{\psi}$ is \textit{regular} if
$$
| V_{\boldsymbol{\psi}, \mathbf{0}} (\mathbb{Z}) \cap [-N,N]^n | \ll N^{n-D_{\boldsymbol{\psi} }  }.
$$
\end{defn}
Similarly as above we also define $V_{\boldsymbol{\psi}, \mathbf{0}} (\mathbb{R})$ to be the set of solutions in $\mathbb{R}^n$
of the equations $\boldsymbol{\psi}(\mathbf{x}) = \mathbf{0}$.

The following is one of the main results of \cite{S} which provides a sufficient condition for a system of polynomials to be regular.
\begin{thm}[Schmidt, \cite{S}]
\label{Schmidt main}
Let $d>1$. Let $\boldsymbol{\psi} = (\boldsymbol{\psi}^{(d)}, \ldots, \boldsymbol{\psi}^{(2)})$ be a system of rational polynomials
with notation as in Definition \ref{def regular},
and also let $\mathbf{f}^{(i)}$ be the system of degree $i$ portion of the polynomials $\boldsymbol{\psi}^{(i)} \ (2 \leq i \leq d)$.
We denote $r_i = | \boldsymbol{\psi}^{(i)} | = | \mathbf{f}^{(i)} | \ (2 \leq i \leq d)$, and $R_{\boldsymbol{\psi} }  = \sum_{i=2}^d r_i$.
If we have
$$
h(  \mathbf{f}^{(i)} )  
\geq d \ 2^{4 i}  (i!)  r_i R_{\boldsymbol{\psi} }  \ \ (2 \leq i \leq d),
$$
then the system $\boldsymbol{\psi}$ is regular.
\end{thm}
Let us denote
\begin{equation}
\label{def rho}
\rho_{d,i}(t) =  d \ 2^{4 i}  (i!)  t^2  \ \ (2 \leq i \leq d)
\end{equation}
so that for each $2 \leq i \leq d$, we have $\rho_{d,i}(t)$ is an increasing function, and
$$
\rho_{d,i}(R_{\boldsymbol{\psi} } ) \geq d \ 2^{4 i}  (i!)  r_i R_{\boldsymbol{\psi} } .
$$

Note Theorem \ref{Schmidt main} is regarding a system of polynomials that does not contain any linear polynomials.
We prove Corollary \ref{cor Schmidt} for systems that contain linear forms as well.
Note the content of the following Corollary \ref{cor Schmidt} is essentially \cite[Corollary 3]{CM}.
\begin{cor}
\label{cor Schmidt}
Let $d>1$. Let $\boldsymbol{\psi} = (\boldsymbol{\psi}^{(d)}, \ldots, \boldsymbol{\psi}^{(1)})$ be a system of rational polynomials with notation as in Definition \ref{def regular}.
Suppose $\boldsymbol{\psi}^{(1)}$ only contains linear forms and that they are linearly independent over $\mathbb{Q}$.
We also let $\mathbf{f}^{(i)}$ be the system of degree $i$ portion of the polynomials $\boldsymbol{\psi}^{(i)} \ (1 \leq i \leq d)$.
We denote $r_i = | \boldsymbol{\psi}^{(i)} | = | \mathbf{f}^{(i)} | \ (1 \leq i \leq d)$, and $R_{\boldsymbol{\psi} }  = \sum_{i=1}^d r_i$.
For each $2 \leq i \leq d$, let $\rho_{d,i}(\cdot)$ be as in ~(\ref{def rho}).
If we have
$$
h(  \mathbf{f}^{(i)} ) \geq \rho_{d,i} (R_{\boldsymbol{\psi} }  - r_1) + r_1 \ \ (2 \leq i \leq d),
$$
then the system $\boldsymbol{\psi}$ is regular.
\end{cor}

\begin{proof}
We have $\boldsymbol{\psi}^{(1)} = \mathbf{f}^{(1)} = \{f_1^{(1)}, \ldots, f_{r_1}^{(1)} \}$. Let
$$
f_i^{(1)} = a_{i1}x_1 + \cdots + a_{in}x_n \ (1 \leq i \leq r_1),
$$
and denote the coefficient matrix of these linear forms to be $A = [a_{ij}]_{1\leq i \leq r_1, 1 \leq j \leq n}$.
Let $\mathbf{e}_{j}$ be the $j$-th standard basis of $\mathbb{R}^n \ (1 \leq j \leq n)$.
Since the linear forms $f_1^{(1)}, \ldots, f_{r_1}^{(1)}$ are linearly independent over $\mathbb{Q}$,
we can find an invertible linear transformation $T \in \ GL_n(\mathbb{Q})$,
where every entry of the matrix is in $\mathbb{Z}$, such that
$(f^{(1)}_i \circ T^{-1})(\mathbf{x}) = m_{n-i + 1 } x_{n-i+1}$, where $m_{n-i + 1 } \in \mathbb{Q} \backslash \{0\}  \ (1 \leq i \leq r_1)$.
For simplicity, let us denote $\mathbf{x}' = (x_{n-r_1+1}, \ldots, x_n)$.
Let
$$
Y = V_{\mathbf{f}^{(1)}, \mathbf{0} } (\mathbb{R}) = \{ \mathbf{x} \in \mathbb{R}^n : \mathbf{f}^{(1)}(\mathbf{x}) = \mathbf{0} \} = \{ \mathbf{x} \in \mathbb{R}^n : A \circ \mathbf{x} = \mathbf{0} \} = \text{Ker}(A ),
$$
which is a subspace of codimension $r_1$. Since $T(Y) = \text{Ker}(A \circ T^{-1})$, it follows from our choice of $T \in \ GL_n(\mathbb{Q})$ that
$$
T(Y) = \mathbb{R} \mathbf{e}_1 + \cdots + \mathbb{R} \mathbf{e}_{n - r_1}.
$$
We also know there exit
$c' , C' > 0$ such that
$$
[-c'N, c'N]^n \subseteq T( [-N, N]^n) \subseteq   [-C'N,C'N]^n.
$$
Define $\boldsymbol{\psi}' = (\boldsymbol{\psi'}^{(d)}, \ldots, \boldsymbol{\psi'}^{(1)}) = \boldsymbol{\psi} \circ T^{-1}$, and let $\mathbf{f'}^{(i)}$ be the system of degree $i$ portion of the polynomials $\boldsymbol{\psi'}^{(i)} \ (1 \leq i \leq d)$. We then have $\mathbf{f'}^{(i)} = \mathbf{f}^{(i)} \circ T^{-1}$.
We can also verify that $V_{\boldsymbol{\psi}', \mathbf{0} } (\mathbb{R}) = T ( V_{\boldsymbol{\psi}, \mathbf{0} } (\mathbb{R}) ).$
Therefore, we obtain
$$
T( V_{\boldsymbol{\psi}, \mathbf{0} }(\mathbb{R}) \cap [-N, N]^n )  \subseteq V_{\boldsymbol{\psi}', \mathbf{0} }(\mathbb{R}) \cap [-C' N, C' N]^n,
$$
and since every entry of the matrix $T \in \ GL_n(\mathbb{Q})$ is in $\mathbb{Z}$, it follows that
\begin{eqnarray}
| V_{\boldsymbol{\psi}, \mathbf{0} } (\mathbb{Z})\cap [-N, N]^n   |
&\leq&  |V_{\boldsymbol{\psi}', \mathbf{0} }(\mathbb{Z}) \cap [-C' N, C' N]^n  |.
\label{eqn 1 in cor 3.3}
\end{eqnarray}
Let $\boldsymbol{\psi''} = (\boldsymbol{\psi'}^{(d)} |_{{\mathbf{x}' = \mathbf{0}}}, \ldots, \boldsymbol{\psi'}^{(2)}|_{{\mathbf{x}' = \mathbf{0}}} )$.
Since $\boldsymbol{\psi' }^{(1)} = \mathbf{0}$ is equivalent to $\mathbf{x}' = \mathbf{0}$, we have
\begin{equation}
\label{eqn 1' in cor 3.3}
|V_{\boldsymbol{\psi}', \mathbf{0} }(\mathbb{Z}) \cap [-C' N, C' N]^n  |
= |V_{\boldsymbol{\psi}'', \mathbf{0} }(\mathbb{Z}) \cap [-C' N, C' N]^{n-r_1}  |.
\end{equation}

Since the degree $i$ portion of $\boldsymbol{\psi'}^{(i)} |_{\mathbf{x}' = \mathbf{0}}$ is $\mathbf{f'}^{(i)} |_{{\mathbf{x}' = \mathbf{0}}}$ for each $2 \leq i \leq d$, we have by Lemma \ref{h ineq 1'} that
\begin{eqnarray}
\notag
h( \mathbf{f'}^{(i)} |_{{\mathbf{x}' = \mathbf{0}}} ) \geq   h( \mathbf{f'}^{(i)} ) - r_1 = h( \mathbf{f}^{(i)} ) - r_1 \geq \rho_{d,i}(R_{\boldsymbol{\psi}} - r_1).
\end{eqnarray}
Thus, it follows by Theorem \ref{Schmidt main} that
\begin{equation}
\label{eqn 2 in cor 3.3}
|V_{\boldsymbol{\psi}'', \mathbf{0} } (\mathbb{Z}) \cap [-C' N, C' N]^{n-r_1}  |
\ll N^{(n-r_1) - \sum_{i=2}^d i r_i }.
\end{equation}
Therefore, we obtain from ~(\ref{eqn 1 in cor 3.3}),  ~(\ref{eqn 1' in cor 3.3}), and ~(\ref{eqn 2 in cor 3.3}) that
\begin{eqnarray}
| V_{\boldsymbol{\psi}, \mathbf{0} }(\mathbb{Z}) \cap [-N, N]^n  | 
&\ll&
N^{ n - \sum_{i=1}^d i r_i }.
\notag
\end{eqnarray}

\end{proof}


Given $\mathbf{g} = \{g_1, \ldots, g_{r_d}  \} \subseteq \mathbb{Q}[x_1, \ldots, x_n]$, a system of forms of degree $d$, and a partition of variables $\mathbf{x} = (\mathbf{y}, \mathbf{z})$, we denote $\overline{\mathbf{g}}$ to be the system obtained by removing all the forms of $\mathbf{g}$ that depend only on the $\mathbf{z}$ variables. Clearly, if we have the trivial partition $\mathbf{x} = (\mathbf{y}, \mathbf{z})$, where $\mathbf{z} = \emptyset$, then $\overline{\mathbf{g}} = \mathbf{g}$.
For a form $g(\mathbf{x})$ over $\mathbb{Q}$, we define $h(g;\mathbf{z})$ to be the smallest number $h_0$ such that $g(\mathbf{x})$
can be expressed as
$$
g(\mathbf{x}) =  g(\mathbf{y}, \mathbf{z}) = \sum_{i=1}^{h_0} u_i v_i + w_0(\mathbf{z}),
$$
where $u_i, v_i$ are rational forms of positive degree $(1 \leq i \leq h_0)$, and
$w_0(\mathbf{z})$ is a rational form only in the $\mathbf{z}$ variables.
We also define $h(\mathbf{g}; \mathbf{z})$ to be
$$
h(\mathbf{g}; \mathbf{z}) = \min_{\boldsymbol{\lambda} \in \mathbb{Q}^{r_d} \backslash \{ \boldsymbol{0} \}}  h( \lambda_{1} g_1 + \cdots + \lambda_{r_d} g_{r_d}; \mathbf{z} ).
$$
If we have the trivial partition, then clearly we have $h(\mathbf{g}; \emptyset) = h(\mathbf{g}).$
We have the following lemma.
\begin{lem}[Lemma 2, \cite{CM}]
\label{Lemma 2 in CM}
Let $\mathbf{g} = \{g_1, \ldots, g_{r_d}  \} \subseteq \mathbb{Q}[x_1, \ldots, x_n]$ be a system of forms of degree $d$, and suppose we have a partition of variables $\mathbf{x} = (\mathbf{y}, \mathbf{z})$. Let $\mathbf{y}'$ be a distinct set of variables with the same number of variables as $\mathbf{y}$.
Then we have
$$
h(\mathbf{g}(\mathbf{y}, \mathbf{z}),  \mathbf{g}(\mathbf{y}', \mathbf{z}) ; \mathbf{z} ) = h( \mathbf{g} ; \mathbf{z} ).
$$
\end{lem}

Given a system of forms which may not be regular, we want to obtain a regular system in a controlled manner. 
The process in the following proposition is referred to as the \emph{regularization} of systems in \cite{CM}, and it
is a crucial component of their method.
Given a system of rational forms $\textbf{f}$, via the regularization process we obtain
another system $\mathcal{R}( \mathbf{f})$ which is regular,
the number of forms it contains is controlled, and its level sets partition the level sets of $\textbf{f}$.
We remark that condition $(3)$ of Proposition \ref{prop reg par}, with a suitable choice of $\boldsymbol{\mathcal{F}}$,
together with Corollary \ref{cor Schmidt} implies that the resulting system is regular.

\begin{prop}[Propositions 1 and 1', \cite{CM}]
\label{prop reg par}
Let $d>1$, and let $\boldsymbol{\mathcal{F}}$ be any collection of non-decreasing functions $\mathcal{F}_i: \mathbb{Z}_{\geq 0} \rightarrow \mathbb{Z}_{\geq 0} \ (2 \leq i \leq d)$. For a collection of non-negative integers $r_1, \ldots, r_d$, there exist constants
$$
C_1(r_1, \ldots, r_d, \boldsymbol{\mathcal{F}} ), \ldots , C_d(r_1, \ldots, r_d,  \boldsymbol{\mathcal{F}} )
$$
such that the following holds.

Given a system of integral forms $\mathbf{f} = ({\mathbf{f}}^{(d)}, \ldots, {\mathbf{f}}^{(1)}) \subseteq \mathbb{Z}[x_1, \ldots, x_n]$, where each
$\mathbf{f}^{(i)}$ is a system of $r_i$ forms of degree $i \ (1 \leq i \leq d)$, and a partition of variables $\mathbf{x} = (\mathbf{y}, \mathbf{z})$, there exists a system of forms
$\mathcal{R}( \mathbf{f}) = ({\mathbf{a}}^{(d)}, \ldots, {\mathbf{a}}^{(1)})$ satisfying the following.
Let $r'_i = |\mathbf{a}^{(i)}| \ (1 \leq i \leq d)$, and $R' = r'_1 + \cdots + r'_d$.
\newline

(1) Each form of the system $\mathbf{f}$ can be written as a rational polynomial expression in the forms of the system $\mathcal{R}( \mathbf{f})$.
In particular, the level sets of $\mathcal{R}( \mathbf{f})$ partition those of $\mathbf{f}$.

(2) For each $1 \leq i \leq d$, $r'_i$ is at most $C_i(r_1, \ldots, r_d, \boldsymbol{\mathcal{F}})$.

(3) The subsystem $({\mathbf{a}}^{(d)}, \ldots, { \mathbf{a} }^{(2)})$ satisfies $h({\mathbf{a}}^{(i)}) \geq \mathcal{F}_i(R')$ for each $2 \leq i \leq d$. Moreover, the linear forms of subsystem ${ \mathbf{a} }^{(1)}$ are linearly independent over $\mathbb{Q}$.

(4) Let $\overline{\mathbf{a}}^{(i)}$ be the system obtained by removing from ${\mathbf{a}}^{(i)}$ all forms that depend only
on the $\mathbf{z}$ variables $(2 \leq i \leq d)$. Then the subsystem $({ \overline{\mathbf{a}}}^{(d)}, \ldots, { \overline{\mathbf{a}} }^{(2)})$
satisfies $h({\overline{\mathbf{a}} }^{(i)} ; \mathbf{z}) \geq \mathcal{F}_i(R')$ for each $2 \leq i \leq d$.
\end{prop}
We will be utilizing this proposition 
in Section \ref{section minor arc} to obtain the minor arc estimate.

\section{Technical Estimates}
\label{technical estimate}

In this section, we provide results from \cite{S} related to Weyl differencing
that are necessary in obtaining estimates for the singular series in Section \ref{section singular series}.
The work here is similar to that of \cite{S}, which is in terms of forms instead of polynomials as in this section. It is stated in
\cite{S} with some explanation that similar results for polynomials also follow, but the details are not shown. We chose to present the necessary details in order to make certain dependency of the constants explicit, which are crucial in our estimates. Let us denote $\mathfrak{B}_1 = [-1,1]^n$. We shall refer to $\mathfrak{B} \subseteq \mathbb{R}^n$ as a box, if $\mathfrak{B}$ is of the form
$$
\mathfrak{B} = I_1 \times \cdots \times I_n,
$$
where each $I_j$ is a closed or open or half open/closed interval $(1 \leq j \leq n)$.
Given a function $G(\mathbf{x})$, we define
$$
\Gamma_{d, G} (\mathbf{x}_1, \ldots, \mathbf{x}_d) = \sum_{t_1=0}^1 \cdots \sum_{t_d=0}^1 (-1)^{t_1 + \cdots + t_d} \
G( t_1 \mathbf{x}_1 + \cdots + t_d \mathbf{x}_d ).
$$
Then it follows that $\Gamma_{d, G}$ is symmetric in its $d$ arguments, and that
$\Gamma_{d, G} (\mathbf{x}_1, \ldots,\mathbf{x}_{d-1}, \mathbf{0}) = 0$ \cite[Section 11]{S}.
It is clear from the definition that
$\Gamma_{d, G} + \Gamma_{d, G'} = \Gamma_{d, G + G'}.$
We also have that if $G$ is a form of degree $j$, where $d > j > 0$, then $\Gamma_{d, G}= 0$ \cite[Lemma 11.2]{S}.

For $\alpha \in \mathbb{R}$, let $\| \alpha \|$ denote the
distance from $\alpha$ to the closest integer. Given $\boldsymbol{\alpha} = (\alpha_1, \ldots, \alpha_n) \in \mathbb{R}^n$,
we let
$$
\| \boldsymbol{\alpha} \| = \max_{1 \leq i \leq n} \| \alpha_i \|.
$$

\begin{lem}\cite[Lemma 13.1]{S}
\label{Lemma 13.1 in S}
Suppose $G(\mathbf{x}) = G^{(0)} + G^{(1)}(\mathbf{x} ) + \cdots + G^{(d)}(\mathbf{x} )$, where $G^{(j)}$ is a form of degree $j$
with real coefficients $(1 \leq j \leq d)$, and $G^{(0)} \in \mathbb{R}$. Let $\mathfrak{B}$ be a box with sides $\leq 1$, let $P > 1$, and put
$$
S' = S'( G, P , \mathfrak{B} ) = \sum_{\mathbf{x} \in P \mathfrak{B} \cap \mathbb{Z}^n } e(G(\mathbf{x})).
$$
Let $\mathbf{e}_1, \ldots , \mathbf{e}_{n}$ be the standard basis vectors of $\mathbb{R}^n$.
Then for any $\varepsilon > 0$, 
we have
$$
|S'|^{2^{d-1}} \ll P^{ (2^{d-1} -d )n + \varepsilon} \sum \left(
\prod_{i=1}^n \min ( P, \|   \Gamma_{d, G^{(d)}} ( \mathbf{x}_1, \ldots , \mathbf{x}_{d-1}, \mathbf{e}_i  ) \|^{-1}   )
 \right),
$$
where the sum $\sum$ is over $(d-1)$-tuples of integer points $\mathbf{x}_1, \ldots , \mathbf{x}_{d-1}$ in $P \mathfrak{B}_1$,
and the implicit constant in $\ll$ depends only on $n,d,$ and $\varepsilon$.
\end{lem}

\begin{lem}\cite[Lemma 14.2]{S}
\label{Lemma 14.2 in S}
Make all the assumptions of Lemma \ref{Lemma 13.1 in S}. Suppose further that
$$
|S'| \geq P^{n-Q}
$$
where $Q>0.$
Let $0 < \eta \leq 1.$
Then the number $N(\eta)$ of integral $(d-1)$-tuples
$$
\mathbf{x}_1, \ldots, \mathbf{x}_{d-1} \in P^{\eta} \mathfrak{B}_1
$$
with
$$
\|  \Gamma_{d, G^{(d)}} ( \mathbf{x}_1, \ldots , \mathbf{x}_{d-1}, \mathbf{e}_i  ) \| < P^{-d + (d-1) \eta } \ (i=1,\ldots, n)
$$
satisfies
$$
N(\eta) \gg P^{ n(d-1)\eta - 2^{d-1} Q - \varepsilon },
$$
where the implicit constant in $\gg$ depends only on $n,d, \eta,$ and $\varepsilon$.
\end{lem}

Let $\boldsymbol{\psi} = \{ \psi_1, \ldots, \psi_{r_d}  \}$ be a system of rational polynomials of degree $d$.
Let $\mathbf{f} = \{ f_1, \ldots, f_{r_d} \}$ be the system of forms, where $f_i$ is the degree $d$ portion of
$\psi_i \ (1 \leq i \leq {r_d} )$.
We define the following exponential sum associated to  $\boldsymbol{\psi}$ and $\mathfrak{B}$,
\begin{equation}
\label{def of S 1}
S( \boldsymbol{\alpha}) = S( \boldsymbol{\psi},  \mathfrak{B} ;\boldsymbol{\alpha}) = \sum_{\mathbf{x} \in P \mathfrak{B} \cap \mathbb{Z}^n} e( \boldsymbol{\alpha}  \cdot \boldsymbol{\psi} (\mathbf{x}) ).
\end{equation}

Let $\mathbf{e}_1, \ldots, \mathbf{e}_n$ be the standard basis vectors of $\mathbb{C}^n$.
We define $\mathfrak{M}_d = \mathfrak{M}_d (\mathbf{f}^{}) $ to be the set of $(d-1)$-tuples $(\mathbf{x}_1, \ldots, \mathbf{x}_{d-1} ) \in (\mathbb{C}^n)^{d-1}$ for which the matrix
$$
[m_{ij}] = [ \Gamma_{d, f_{j}} ( \mathbf{x}_1, \ldots , \mathbf{x}_{d-1}, \mathbf{e}_i ) ] \ \ \ \  (1 \leq j \leq r_d, 1\leq i \leq n)
$$
has rank strictly less than $r_d$. For $R>0$, we denote $z_{R} (\mathfrak{M}_d)$ to be the number of integer points $(\mathbf{x}_1, \ldots, \mathbf{x}_{d-1} )$ on
$\mathfrak{M}_d$ such that $$\max_{1 \leq i \leq d-1} \max_{1 \leq j \leq n}  | x_{ij} | \leq R,$$
where $\mathbf{x}_i = (x_{i1}, \ldots, x_{in}) \ (1 \leq i \leq d-1)$.

Let $P > 1$, $Q > 0$, and $\varepsilon >0$ be given, and suppose that $d > 1$. We then have:
\begin{lem}\cite[Lemma 15.1]{S}
\label{Lemma 15.1 in S}
Given a box $\mathfrak{B}$ with sides $\leq 1$, define the sum  $S( \boldsymbol{\alpha})$ associated with $\boldsymbol{\psi}$ and $\mathfrak{B}$ as in ~(\ref{def of S 1}). Given $0 < \eta \leq 1$,
one of the following three alternatives must hold:

(i) $|  S( \boldsymbol{\alpha})  | \leq P^{n-Q}$.

(ii) there exists $n_0 \in \mathbb{N}$ such that
$$
n_0 \ll P^{r_d(d-1) \eta} \text{  and  } \|  n_0  \boldsymbol{\alpha} \| \ll P^{ -d + r_d (d-1) \eta}.
$$

(iii) $ z_{R} (\mathfrak{M}_d) \gg R^{ (d-1)n - 2^{d-1}(Q/ \eta) - \varepsilon}$
holds with $R = P^{\eta}$.
\newline
\newline
All implicit constants 
depend at most on
$n,d, r_d, \eta, \varepsilon$ and $\mathbf{f}$.
\end{lem}

\begin{proof}
Take $\boldsymbol{\alpha} \in \mathbb{R}^{r_d}$. Let
$\boldsymbol{\alpha} \cdot \boldsymbol{\psi}(\mathbf{x}) = G^{(0)}+ G^{(1)}(\mathbf{x}) + \cdots +  G^{(d)}(\mathbf{x})$, where
$G^{(j)}$ is a form of degree $j \ (1 \leq  j \leq d)$, and $G^{(0)} \in \mathbb{R}$.
Suppose $(i)$ fails, then we may apply Lemma \ref{Lemma 14.2 in S}. The number
$N(\eta)$ of integral $(d-1)$-tuples $\mathbf{x}_1$, \ldots, $\mathbf{x}_{d-1}$ in $P^{\eta} \mathfrak{B}_1$ with
\begin{equation}
\label{(15.2) in S}
\|   \Gamma_{d, G^{(d)}} ( \mathbf{x}_1, \ldots , \mathbf{x}_{d-1}, \mathbf{e}_i  ) \| < P^{-d + (d-1) \eta } \ (i=1,\ldots, n)
\end{equation}
satisfies
$$
N(\eta) \gg R^{n(d-1) - 2^{d-1} (Q/ \eta) - \varepsilon },
$$
where $R = P^{\eta}$, and  the implicit constant in $\gg$ depends only on $n,d, \eta,$ and $\varepsilon$.

Recall $\boldsymbol{\psi} = \{ \psi_1, \ldots,  \psi_{r_d} \}$. Given $\mathbf{x}_1$, \ldots, $\mathbf{x}_{d-1}$
as above, we form the matrix
$$
[m_{ij}]_{\mathbf{x}_1, \ldots , \mathbf{x}_{d-1}} = [ \Gamma_{ d, \psi_j } ( \mathbf{x}_1, \ldots, \mathbf{x}_{d-1}, \mathbf{e}_i  ) ]
  \  \  (1 \leq i \leq n, 1 \leq j \leq r_d).
$$
Recall $f_j$ is the degree $d$ portion of $\psi_j \ (1 \leq j \leq r_d)$, and
$\mathbf{f}= \{ f_1, \ldots,  f_{r_d} \}$.
Since each $\psi_j$ is of degree $d$, it follows that $\Gamma_{d, \psi_j} = \Gamma_{d, f_j} \  (1 \leq j \leq r_d)$.
It is also clear that $G^{(d)}( \mathbf{x} ) = \boldsymbol{\alpha} \cdot \mathbf{f}(\mathbf{x})$.
Now if this matrix $[m_{ij}]_{\mathbf{x}_1, \ldots , \mathbf{x}_{d-1}}$ has rank less than $r_d$ for each of the $(d-1)$-tuples counted by $N(\eta)$, then by the definition of
$z_R(\mathfrak{M}_d)$ we have that
$$
z_R(\mathfrak{M}_d) \geq N(\eta) \gg R^{n(d-1) - 2^{d-1} (Q/ \eta) - \varepsilon },
$$
where again the implicit constant in $\gg$ depends only on $n,d, \eta,$ and $\varepsilon$.
Thus we have $(iii)$ in this case. Hence, we may suppose that at least one of these matrices, which we denote by $[m_{ij}]$, has rank $r_d$.
Without loss of generality, suppose the submatrix $M_0$ formed by taking the first $r_d$ columns of $[m_{ij}]$
has rank $r_d$. Let $n_0 = \det (M_0)$.

It follows from the definition of $\Gamma_{d, f_j}$ that every monomial occurring in
$\Gamma_{d, f_j} (\mathbf{x}_1, \ldots, \mathbf{x}_d)$ has some component of $\mathbf{x}_i$ as a factor for each $1 \leq i \leq d$ \cite[Proof of Lemma 11.2]{S}.
Also, the maximum absolute value of all coefficients of
$\Gamma_{d, f_j}$ is bounded by a constant dependent only on $d$ and the coefficients of $f_j$ \cite[Lemma 11.3]{S}.
Therefore, by the construction of $[m_{ij}]$ we have
$$
m_{ij} \ll R^{d-1},
$$
and hence
$$
n_0 \ll R^{r_d(d-1)} = P^{r_d(d-1)\eta},
$$
where the implicit constants in $\ll$ depend only on $r_d$ and $\mathbf{f}$.

We have
$$
\Gamma_{d, G^{(d)}} = \sum_{j=1}^{r_d} \Gamma_{d, \alpha_j f_j} = \sum_{j=1}^{r_d} \alpha_j \Gamma_{d, f_j}.
$$
Hence, from ~(\ref{(15.2) in S}) we may write
$$
\sum_{j=1}^{r_d} \alpha_j m_{ij} = c_i + \beta_i \ (1 \leq i \leq n),
$$
where the $c_i$ are integers and the $\beta_i$ are real numbers bounded by the right hand side of ~(\ref{(15.2) in S}).
Let $u_1, \ldots, u_{r_d}$ be the solution of the system of linear equations
\begin{equation}
\label{(15.3) in S}
\sum_{j=1}^{r_d} u_j m_{ij} = n_0 c_i \ (1 \leq i \leq r_d).
\end{equation}
Then
\begin{equation}
\label{(15.4) in S}
\sum_{j=1}^{r_d} (  n_0 \alpha_j - u_j  )m_{ij} = n_0 \beta_i \ (1 \leq i \leq r_d).
\end{equation}
By applying Cram\'{e}r's rule to ~(\ref{(15.3) in S}), it follows that the $u_j$ are integers. Also, by applying Cram\'{e}r's rule  to ~(\ref{(15.4) in S}),
we obtain that
\begin{eqnarray}
\| n_0 \alpha_j \| \leq | n_0 \alpha_j - u_j | \ll R^{ (d-1)(r_d - 1) }  P^{-d + (d-1) \eta} = P^{ - d + (d-1) r_d \eta } ,
\end{eqnarray}
where the implicit constant in $\ll$ depends only on $r_d$ and $\mathbf{f}$.
This completes the proof of Lemma \ref{Lemma 15.1 in S}
\end{proof}

We define $g_d( \mathbf{f} )$
to be the largest real number such that
\begin{equation}
\label{def gd}
z_P(\mathfrak{M}_d) \ll P^{n(d-1) - g_d( \mathbf{f} ) + \varepsilon}
\end{equation}
holds for each $\varepsilon >0$. It was proved in \cite[pp. 280, Corollary]{S} that
\begin{equation}
\label{h and g}
h(\mathbf{f}) < \frac{d!}{  (\log 2)^{d} }  \left( g_d( \mathbf{f} ) + (d-1)r_d (r_d - 1)  \right).
\end{equation}
Let
$$
\gamma_d = \frac{2^{d-1} (d-1) r_d}{ g_d( \mathbf{f} ) }
$$
when $g_d( \mathbf{f} ) > 0$.
We let $\gamma_d = + \infty$ if $g_d( \mathbf{f} ) = 0$.
We also define
\begin{equation}
\label{def gamma'}
\gamma'_d = \frac{ 2^{d-1} }{ g_d( \mathbf{f} ) } = \frac{ \gamma_d }{ (d-1) r_d }.
\end{equation}
\begin{cor}\cite[pp.276, Corollary]{S}
\label{cor 15.1 in S}
Given a box $\mathfrak{B}$ with sides $\leq 1$, we define the sum  $S( \boldsymbol{\alpha}) $ associated with $\boldsymbol{\psi}$ and $\mathfrak{B}$ as in ~(\ref{def of S 1}).
Suppose $\varepsilon' > 0$ is sufficiently small and $Q>0$ satisfies
$$
Q \gamma'_d < 1.
$$
Then one of the following alternatives must hold:

(i) $|  S( \boldsymbol{\alpha})  | \leq P^{n-Q}$.

(ii) there exists $n_0 \in \mathbb{N}$ such that
$$
n_0 \ll P^{Q \gamma_d + \varepsilon'} \text{  and  } \|  n_0 \boldsymbol{\alpha} \| \ll P^{ -d + Q \gamma_d + \varepsilon'},
$$
where the implicit constants in $\ll$ depend only on $n, d, r_d, \varepsilon', Q$ and $\mathbf{f}$.
\end{cor}
Note the fact that the implicit constant depends on $\mathbf{f}$, but not
on other lower order terms of $\boldsymbol{\psi}$ is an important feature which
we make use of in Section \ref{section singular series}.

\begin{proof}
Since $ Q \gamma'_d < 1$, we can choose $\varepsilon_1 > 0$ sufficiently small so that
$\eta = Q \gamma'_d + \varepsilon_1$ satisfies $0 < \eta \leq 1$.
Also with this choice of $\eta$, we have
\begin{eqnarray}
\frac{2^{d-1} Q} { \eta} &=& \frac{ 2^{d-1} Q}{ Q \gamma'_d + \varepsilon_1 }
= \frac{ g_d( \mathbf{f} )}{ 1+ \varepsilon_1 g_d(\mathbf{f} ) / ( 2^{d-1} Q  ) }
\notag
< g_d( \mathbf{f} ).
\notag
\end{eqnarray}
Then choose $\varepsilon_0 > 0$ such that $2^{d-1}Q / \eta + \varepsilon_0 < g_d( \mathbf{f} )$. By the definition of $g_d( \mathbf{f} )$ we have
$$
z_R(\mathfrak{M}_d) \ll R^{n(d-1) - 2^{d-1}Q / \eta - \varepsilon_0}.
$$
Thus in this case we see that the statement $(iii)$ in Lemma \ref{Lemma 15.1 in S}
can not occur with $0 < \varepsilon < \varepsilon_0$.
Also the equation $\eta = Q \gamma'_d + \varepsilon_1$ implies
$$
r_d (d-1) \eta = Q \gamma_d +  r_d (d-1) \varepsilon_1.
$$
Therefore, from Lemma \ref{Lemma 15.1 in S} (applying it with $0 < \varepsilon < \varepsilon_0$) we obtain our result with
$\varepsilon' = r_d (d-1) \varepsilon_1$.
\end{proof}

For the rest of this section, we assume $\boldsymbol{\psi}$ to be a system of integral polynomials of degree $d$.
When the polynomials $\boldsymbol{\psi}$ in question are over $\mathbb{Z}$, we consider the following.

Hypothesis ($\star$). Let $\mathfrak{B}$ be a box in $\mathbb{R}^n$. For any $\Delta > 0$,
there exists $P_1 = P_1(\mathbf{f}, \Omega,  \Delta, \mathfrak{B})$ such that for $P > P_1$,
each $\boldsymbol{\alpha} \in \mathbb{T}^{r_d}$ satisfies at least one of the following two alternatives. Either

(i)  $|S (\boldsymbol{\alpha}) | \leq P^{n - \Delta \Omega},$ or

(ii) there exists  $q = q(\boldsymbol{\alpha}) \in \mathbb{N}$ such that  
$$
q \leq P^{\Delta} \  \text{  and  } \  \|  q \boldsymbol{\alpha}  \| \leq P^{-d + \Delta}. 
$$

We will say that the restricted Hypothesis ($\star$) holds if the above condition holds for each $\Delta$ in $0 < \Delta \leq 1$.

The important thing to note here is that the lower bound for $P$ in Hypothesis ($\star$) only depends on $\mathbf{f}$, and not on $\boldsymbol{\psi}$. In other words, only the highest degree portion of the polynomials $\boldsymbol{\psi}$ play a role in this estimate.

\begin{prop}\cite[Proposition $\text{II}_0$]{S}
\label{prop omega and g}
Given a box $\mathfrak{B}$ with sides $\leq 1$, Hypothesis $(\star)$ is true for any $\Omega$ in
\begin{equation}
\label{(10.6) in S}
0 < \Omega < \frac{g_d( \mathbf{f} )}{2^{d-1} (d-1) r_d}.
\end{equation}

\end{prop}

\begin{proof}
It follows from ~(\ref{(10.6) in S}) that
$ \Omega \gamma_d < 1$.
We set $Q = \Delta \Omega$, and let $\varepsilon > 0$ be sufficiently small so that $Q \gamma_d + \varepsilon < \Delta$.
First, we suppose $\Delta \leq (d-1)r_d$. In this case, it follows that $Q \gamma'_d < 1$.
Thus it follows from Corollary \ref{cor 15.1 in S} that there exists $P_0 = P_0( \mathbf{f}, \Omega, \Delta)$ such that whenever $P > P_0$, either

$(i)$ $|S (\boldsymbol{\alpha})| \leq P^{n - \Delta \Omega}$, or

$(ii)$ there exists $q \in \mathbb{N}$ such that
$$
q \leq P^{\Delta} \text{ and } \|  q \boldsymbol{\alpha} \| \leq P^{-d + \Delta}.
$$
On the other hand, if $\Delta > (d-1)r_d$, then the case $(ii)$ above is always true by Dirichlet's Theorem on Diophantine approximation.
\end{proof}

For each $q \in \mathbb{N}$, we denote $\mathbb{U}_q$ as the group of units in $\mathbb{Z}/ q \mathbb{Z}.$ 
Given $\mathbf{m} \in \mathbb{U}_q^{r_d}$, we define
\begin{equation}
\label{def of E}
E(q^{-1} \mathbf{m}) = E( \boldsymbol{\psi}, q ; q^{-1} \mathbf{m} ) = q^{-n} \sum_{ \mathbf{x} \ (\text{mod }q) } e(q^{-1} \ \mathbf{m} \cdot \boldsymbol{\psi} (\mathbf{x}) ).
\end{equation}

\begin{lem}\cite[Lemma 7.1]{S}
\label{bound on E}
Suppose $\Omega$ satisfies ~(\ref{(10.6) in S}).
Then for $0 < Q < \Omega$, we have
\begin{equation}
\label{(7.1) in S}
| E (q^{-1} \mathbf{m} ) | \ll q^{- Q},
\end{equation}
where the implicit constant in $\ll$ depends only on $\mathbf{f}$, $Q$ and $\Omega$.
\end{lem}
Again the fact that the implicit constant depends on $\mathbf{f}$, but not
on other lower order terms of $\boldsymbol{\psi}$ becomes crucial when we apply this lemma in Section \ref{section singular series}.

\begin{proof}
Since $E(q^{-1} \mathbf{m} ) = q^{-n} S( \boldsymbol{\alpha} )$ with $\boldsymbol{\alpha}  = q^{-1} \mathbf{m}$, $P = q$ and $\mathfrak{B} = [0,1)^{r_d}$,
and with our choice of $\Omega$
we know that Hypothesis ($\star$) is satisfied by Proposition \ref{prop omega and g}. Thus we apply it with $\Delta = Q / \Omega < 1$. Let $q$ be sufficiently large, and suppose we are in case $(ii)$ of Hypothesis ($\star$). Then we know there
exists $q_0 \leq q^{\Delta} < q$ (when $q \not = 1$) with
$$
\|  q_0 q^{-1} \mathbf{m}  \| \leq q^{-d + \Delta} < q^{-1}.
$$
Since $(\mathbf{m}, q) = 1$, this is not possible.
Therefore, we must have case $(i)$ of Hypothesis ($\star$), which is precisely the inequality
~(\ref{(7.1) in S}).
\end{proof}

\section{Hardy-Littlewood Circle Method: Minor Arcs}
\label{section circle method}
\label{section minor arc}

For each $q \in \mathbb{N}$, recall we let $\mathbb{U}_q$ be the group of units in $\mathbb{Z}/ q \mathbb{Z}$. When $q=1$ we let $\mathbb{U}_1 = \{0\}$.
Let us denote $\mathbb{T} = \mathbb{R} / \mathbb{Z}$.
For a given value of $C>0$ and an integer $1 \leq q \leq (\log N)^C$, we define the \textit{major arc}
$$
\mathfrak{M}_{{m}, q}(C) = \{  \alpha \in \mathbb{T} : \| \alpha - m / q \| \leq N^{-d} (\log N)^C  \}
$$
for each $m \in \mathbb{U}_q$. Recall $\| \beta \|$ is the distance from $\beta \in \mathbb{R}$ to the nearest integer, which induces a metric on $\mathbb{T}$ via $d(\alpha, \beta) = \| \alpha - \beta \|$.
These arcs are disjoint for $N$ sufficiently large,
and we define
$$
\mathfrak{M}(C) = \bigcup_{q \leq (\log N)^C } \bigcup_{ {m} \in \mathbb{U}_q }  \mathfrak{M}_{{m}, q}(C).
$$
We then define the \textit{minor arcs} to be
$$
\mathfrak{m}(C) = \mathbb{T} \backslash \mathfrak{M}(C).
$$

We obtain the following bound on the minor arcs in this section.
\begin{prop}
\label{prop minor arc bound}
Let $b(\mathbf{x}) \in \mathbb{Z}[x_1, \ldots, x_n]$ be a polynomial of degree $d$. Let $T(b; \alpha)$
be defined as in ~(\ref{def of exp sum T}).
Then there exists a positive number $A_d$ dependent only on $d$ such that the following holds. Suppose $b(\mathbf{x})$
satisfies $h^{\star}(f_b) > A_d$.
Then, given any $c>0$, there exists $C>0$ such that
$$
\int_{\mathfrak{m}(C)}  T(b; \alpha)  \ d {\alpha}  
\ll
\frac{N^{n - d } }{(\log N)^{c}}.
$$
\end{prop}
The proposition is achieved by splitting the exponential sum $T(b; \alpha)$ over certain level sets based on a decomposition of the
polynomial $b(\mathbf{x})$. Thus before we get into the proof of Proposition \ref{prop minor arc bound} we first
establish this decomposition in six steps, where the resulting decomposition is given in (\ref{decom of b last}).
For simplicity, we let $f(\mathbf{x})$ be the degree $d$ portion of $b(\mathbf{x})$ for the remainder of the paper.
We let $h = h(f)$, and let $0 < M < h^{\star}(f) \leq h$ to be chosen later.

\subsection*{Step 1: Decomposition of the variables} As explained in the paragraph before ~(\ref{h-inv decomp after linear transfn}),
by relabeling the variables if necessary we have
$$
f = (x_1 + \ell_1) v'_1 + \cdots + (x_M + \ell_M) v'_M + u'_{M+1}v'_{M+1} + \cdots + u'_h v'_h,
$$
where each $\ell_i$ is a linear form in $\mathbb{Q}[x_{M+1}, \ldots, x_n] \ (1 \leq i \leq M)$,
and each $u'_{i'}$ and $v'_j$ are rational forms of positive degree $(M+1 \leq i' \leq h, 1 \leq j \leq h)$.
We can then find a monomial $x_{i_1} x_{i_2} \cdots x_{i_d}$, where $M < i_1 \leq i_2 \leq \cdots \leq i_d$, of $f$ with a non-zero coefficient.
This is the case, for otherwise it means that every monomial of $f$ is divisible by one of $x_1, \ldots, x_M$, and consequently that
$h = h(f)\leq M$, which is a contradiction.
We denote the distinct variables of $\{ x_{i_1}, x_{i_2}, \ldots , x_{i_d} \} \subseteq \{x_{M+1}, \ldots, x_n \}$
by $\{w_1, \ldots, w_K \}$, and let $\mathbf{w} = (w_1, \ldots, w_K)$. Clearly, we have $K \leq d$.
We selected these $K$ variables for the purpose of applying Weyl differencing later.
We also label $\mathbf{y} = (x_1,\ldots, x_M) = (y_1, \ldots, y_M)$ for notational convenience,
let $\mathbf{z} = \{ x_{M+1}, \ldots, x_n \} \backslash \mathbf{w}$, and denote $\mathbf{z} = (z_1, \ldots, z_{n - M - K})$.
We note that each $\ell_i$ is a rational linear form only in the $\mathbf{w}$ and the $\mathbf{z}$ variables $(1 \leq i \leq h)$.

\subsection*{Step 2: Decomposition of $f(\mathbf{x})$} We define $g_M$ with respect to $f$ as in ~(\ref{def gM}).
By Lemma \ref{h ineq2}, we have
\begin{equation}
\label{main prop eqn 2 - 2}
{f}(\mathbf{x}) = f( \mathbf{w}, \mathbf{y}, \mathbf{z}) = g_M( \mathbf{w}, \mathbf{y}, \mathbf{z} ) + {f} (\mathbf{w}, (- \ell_1, \ldots, -\ell_M), \mathbf{z}),
\end{equation}
where
\begin{equation}
\label{main prop eqn 1 - 2}
h_{  }(g_M( \mathbf{w}, \mathbf{y}, \mathbf{z} ) ) \geq M \ \ \text{ and } \ \ h_{  } ({f} (\mathbf{w}, (- \ell_1, \ldots, -\ell_M), \mathbf{z})) = h-M.
\end{equation}
We then have
\begin{equation}
\label{gM is 0'}
{f}(\mathbf{0}, \mathbf{y}, \mathbf{z}) = g_M( \mathbf{0}, \mathbf{y}, \mathbf{z} ) + {f} (\mathbf{0}, (- \ell_1 |_{\mathbf{w} = \mathbf{0}}, \ldots, -\ell_M |_{\mathbf{w} = \mathbf{0}}),  \mathbf{z}).
\end{equation}
Let us denote
$$
f_M(\mathbf{z} ) = {f} (\mathbf{0}, (- \ell_1 |_{\mathbf{w} = \mathbf{0}}, \ldots, -\ell_M |_{\mathbf{w} = \mathbf{0}}), \mathbf{z}).
$$
Consequently, we obtain from Lemma \ref{h ineq 1} and ~(\ref{main prop eqn 1 - 2}) that
\begin{equation}
\label{h eqn 1}
h_{  }(g_M( \mathbf{0}, \mathbf{y}, \mathbf{z} ) ) \geq M-K \geq M-d,
\end{equation}
and
\begin{equation}
\label{h eqn 2}
h ( f_M(\mathbf{z} ) ) \geq h - M - K \geq h- M- d.
\end{equation}

\subsection*{Step 3: Decomposition of $b(\mathbf{x})$ with respect to $\mathbf{w}$, $\mathbf{y}$, and $\mathbf{z}$} Let
$$
b_M(\mathbf{z}) = {b}(\mathbf{0}, (- \ell_1 |_{\mathbf{w} = \mathbf{0}}, \ldots, -\ell_M |_{\mathbf{w} = \mathbf{0}}), \mathbf{z} ).
$$
It is clear that the degree $d$ portion of the polynomial ${b}(\mathbf{0}, \mathbf{y}, \mathbf{0})$
is $g_M(\mathbf{0}, \mathbf{y}, \mathbf{0})$.
Let use denote
\begin{eqnarray}
\label{big equation in decomposition1}
&&{b}(\mathbf{0}, \mathbf{y}, \mathbf{z}) - b_M(\mathbf{z})
\\
&=& \sum_{j=1}^{d-1} \ \sum_{1 \leq t_1 \leq \cdots \leq t_{j} \leq M}
\left( \sum_{k=0}^{d-j} \Psi^{(k)}_{t_1, \ldots, t_{j}} ( \mathbf{z} )  \right) y_{t_1} \cdots y_{t_{j}}
+ \left( \sum_{k=1}^{d} \Psi^{(k)}_{\emptyset} ( \mathbf{z} ) \right)  + g_M(\mathbf{0}, \mathbf{y}, \mathbf{0}),
\notag
\end{eqnarray}
where $\Psi^{(k)}_{t_1, \ldots, t_{j}} ( \mathbf{z} )$ and
$\Psi^{(k)}_{\emptyset} ( \mathbf{z} )$ are forms of degree $k$.
With these notations, we have the following decomposition,
\begin{eqnarray}
\label{decomp of b 1}
&&{b}(\mathbf{w}, \mathbf{y}, \mathbf{z})
\\
&=&
{b}(\mathbf{w}, \mathbf{0}, \mathbf{0})
+ \sum_{j=1}^{d-1}  \ \sum_{1 \leq i_1 \leq \cdots \leq i_j \leq K}
\left( \sum_{k=1}^{d-j} \Phi^{(k)}_{i_1, \ldots, i_{j}} ( \mathbf{y}, \mathbf{z} )  \right) w_{i_1} \cdots w_{i_j}
\notag
\\
&+&
\sum_{j=1}^{d-1} \ \sum_{1 \leq t_1 \leq \cdots \leq t_{j} \leq M}
\left( \sum_{k=0}^{d-j} \Psi^{(k)}_{t_1, \ldots, t_{j}} ( \mathbf{z} )  \right) y_{t_1} \cdots y_{t_{j}}
+ \left( \sum_{k=1}^{d} \Psi^{(k)}_{\emptyset} ( \mathbf{z} ) \right)  + g_M(\mathbf{0}, \mathbf{y}, \mathbf{0})
\notag
\\
&+& b_M(\mathbf{z}) - b(\mathbf{0}, \mathbf{0},\mathbf{0}),
\notag
\end{eqnarray}
which we describe below.
We note that $\Phi^{(k)}_{i_1, \ldots, i_{j}} ( \mathbf{y}, \mathbf{z} )$
are forms of degree $k$.
The above decomposition establishes the following. The term
$$
{b}(\mathbf{w}, \mathbf{0}, \mathbf{0})
+ \sum_{j=1}^{d-1}  \ \sum_{1 \leq i_1 \leq \cdots \leq i_j \leq K}
\left( \sum_{k=1}^{d-j} \Phi^{(k)}_{i_1, \ldots, i_{j}} ( \mathbf{y}, \mathbf{z} )  \right) w_{i_1} \cdots w_{i_j}
$$
consists of all the monomials of $b(\mathbf{x})$, which involve any variables of $\mathbf{w}$.
Consequently, we have
\begin{eqnarray}
{b}(\mathbf{0}, \mathbf{y}, \mathbf{z})
&=&
\sum_{j=1}^{d-1} \ \sum_{1 \leq t_1 \leq \cdots \leq t_{j} \leq M}
\left( \sum_{k=0}^{d-j} \Psi^{(k)}_{t_1, \ldots, t_{j}} ( \mathbf{z} )  \right) y_{t_1} \cdots y_{t_{j}}
+ \left( \sum_{k=1}^{d} \Psi^{(k)}_{\emptyset} ( \mathbf{z} ) \right) + g_M( \mathbf{0}, \mathbf{y}, \mathbf{0} )
\notag
\\
&+& b_M(\mathbf{z}),
\notag
\end{eqnarray}
and the degree $d$ portion of $b( \mathbf{0}, \mathbf{y}, \mathbf{z}) = f( \mathbf{0}, \mathbf{y}, \mathbf{z})$.
Clearly, the degree $d$ portion of $b_M(\mathbf{z})$ is
$$
f_M(\mathbf{z}) = {f} (\mathbf{0}, (- \ell_1 |_{\mathbf{w} = \mathbf{0}}, \ldots, -\ell_M |_{\mathbf{w} = \mathbf{0}}),  \mathbf{z}).
$$
It then follows from ~(\ref{gM is 0'}) and ~(\ref{big equation in decomposition1}) that the degree $d$ portion of
$$
\sum_{j=1}^{d-1} \ \sum_{1 \leq t_1 \leq \cdots \leq t_{j} \leq M}
\left( \sum_{k=0}^{d-j} \Psi^{(k)}_{t_1, \ldots, t_{j}} ( \mathbf{z} )  \right) y_{t_1} \cdots y_{t_{j}}
+ \left( \sum_{k=1}^{d} \Psi^{(k)}_{\emptyset} ( \mathbf{z} ) \right)
+ g_M( \mathbf{0}, \mathbf{y}, \mathbf{0} )
$$
is
$$
g_M( \mathbf{0}, \mathbf{y}, \mathbf{z} ) =  \sum_{j=1}^{d-1} \ \sum_{1 \leq t_1 \leq \cdots \leq t_{j} \leq M}
 \Psi^{(d-j)}_{t_1, \ldots, t_{j}} ( \mathbf{z} ) \ y_{t_1} \cdots y_{t_{j}}
+ \Psi^{(d)}_{\emptyset} ( \mathbf{z} )
+ g_M( \mathbf{0}, \mathbf{y}, \mathbf{0} ).
$$
We also know from ~(\ref{gM is 0'}) that $g_M( \mathbf{0}, (- \ell_1 |_{\mathbf{w} = \mathbf{0}}, \ldots, -\ell_M |_{\mathbf{w} = \mathbf{0}}), \mathbf{z} ) = 0$,
and consequently,
\begin{eqnarray}
\Psi^{(d)}_{\emptyset} ( \mathbf{z} )
&=&
 \left. \left( -\sum_{j=1}^{d-1} \ \sum_{1 \leq t_1 \leq \cdots \leq t_{j} \leq M} \Psi^{(d-j)}_{t_1, \ldots, t_{j}} ( \mathbf{z} ) \ y_{t_1} \cdots y_{t_{j}}
\right) \right|_{y_i = - \ell_i |_{\mathbf{w} = \mathbf{0} \ (1 \leq i \leq M)} }
\notag
\\
&&
\notag
\\
&-& g_M( \mathbf{0}, (- \ell_1 |_{\mathbf{w} = \mathbf{0}}, \ldots, -\ell_M |_{\mathbf{w} = \mathbf{0}}), \mathbf{0} ).
\notag
\end{eqnarray}
In other words, $\Psi^{(d)}_{\emptyset} ( \mathbf{z} ) $ can be expressed as a rational polynomial
in the forms $ \{ \Psi^{(d-j)}_{t_1, \ldots, t_{j}} ( \mathbf{z} ) :  1 \leq j \leq d-1, 1 \leq t_1 \leq \cdots \leq t_{j} \leq M \} \cup \{ \ell_i |_{\mathbf{w} = \mathbf{0}} : 1 \leq i \leq M \}$.

\subsection*{Step 4: Regularization of systems $\Phi$ and $\Psi$} We denote by $\Phi = \{ \Phi^{(k)}_{i_1, \ldots, i_{j}} : 1 \leq j \leq d-1, 1 \leq i_1 \leq \cdots \leq i_j \leq K, 1 \leq k \leq d-j \}$.
Note every polynomial of $\Phi$ has degree strictly less than $d$, and involves only the $\mathbf{y}$ and the $\mathbf{z}$ variables.
Clearly, we have $|\Phi| \leq d^2 K^d \leq d^{d+2}$. We apply Proposition \ref{prop reg par} to the system $\Phi$ with respect
to the functions $\boldsymbol{\mathcal{F}} = \{\mathcal{F}_2, \ldots, \mathcal{F}_{d-1} \}$, where $\mathcal{F}_i(t) = {\rho}_{d,i} (2 + 2t) + 2t$ for $2 \leq i \leq d-1$, and obtain
$\mathcal{R}(\Phi) = ( \mathbf{a}^{(d-1)}, \ldots, \mathbf{a}^{(1)} )$.
For each form $a^{({s})}_i \in \mathbf{a}^{({s})} \ (1 \leq s \leq d-1, 1 \leq i \leq | \mathbf{a}^{({s})}|)$, we write
\begin{equation}
\label{defn of a's}
a^{({s})}_i(\mathbf{y}, \mathbf{z}) = \sum_{k=0}^{s} \sum_{1 \leq i_1 \leq \cdots \leq i_k \leq M} \widetilde{\Psi}^{(s - k)}_{{s} :i: i_1, \ldots, i_k}(\mathbf{z}) y_{i_1} \cdots y_{i_k},
\end{equation}
where each
$\widetilde{\Psi}^{(s - k)}_{{s} :i: i_1, \ldots, i_k}(\mathbf{z})$ is a form of degree $s - k$.
Thus each form $a^{({s})}_i$ introduces at most $({s}+1) M^{s} \leq d M^d$ forms in $\mathbf{z}$.
Also for each $1 \leq i \leq d-1$, we denote $\overline{\mathbf{a}}^{(i)}$
to be the system obtained by removing all forms that depend only on the $\mathbf{z}$ variables from $\mathbf{a}^{(i)}$.
Let $\overline{\mathcal{R}}(\Phi)= (\overline{\mathbf{a}}^{(d-1)}, \ldots, \overline{\mathbf{a}}^{(1)} )$, $R_2 = \sum_{i = 1}^{d-1} \ | \overline{\mathbf{a}}^{(i)} |$, and $D_2 = \sum_{i = 1}^{d-1} i \ | \overline{\mathbf{a}}^{(i)} |$. By relabeling if necessary, we denote the
elements of $\overline{\mathbf{a}}^{(s)}$ by $\overline{\mathbf{a}}^{(s)} = \{ a^{(s)}_i : 1 \leq i \leq | \overline{\mathbf{a}}^{(s)} | \}$
for each $1 \leq s \leq d-1$.

Let
\begin{eqnarray}
\Psi &=& \{  \Psi^{(k)}_{t_1, \ldots, t_{j}} (\mathbf{z}): 1 \leq j \leq d-1, 1 \leq t_1 \leq \cdots \leq t_{j} \leq M ,   0 \leq k \leq d-j  \}
\notag
\\
&\cup&  \  \{ {\Psi}^{(k)}_{\emptyset}(\mathbf{z}) : 1 \leq k < d \}
\notag
\\
&\cup&
\notag
\{ \ell_i |_{\mathbf{w} = \mathbf{0}} : 1 \leq i \leq M \}
\\
&\cup&  \  \{ \widetilde{\Psi}^{(s - k)}_{{s} :i: i_1, \ldots, i_k}(\mathbf{z}) :1 \leq {s} \leq d-1, 1 \leq i \leq |\mathbf{a}^{(s)}|, 1 \leq k \leq {s}, 1 \leq i_1 \leq \cdots \leq i_k \leq M  \}.
\notag
\end{eqnarray}
In other words, $\Psi$ is the collection of $\ell_i |_{\mathbf{w} = \mathbf{0}}$, and all $ \Psi^{(k)}_{t_1, \ldots, t_{j}} (\mathbf{z})$,
$\widetilde{\Psi}^{(s - k)}_{{s} :i: i_1, \ldots, i_k}(\mathbf{z})$, and
${\Psi}^{(k)}_{\emptyset}(\mathbf{z})$ except ${\Psi}^{(d)}_{\emptyset}(\mathbf{z})$. In particular, every polynomial of $\Psi$ has degree strictly less than $d$.
We can see that
$$
|\Psi| \leq d^2 M^d + d + M + |\mathcal{R}(\Phi) | d M^d.
$$
We let $\mathcal{R}(\Psi)$ be a regularization of
$\Psi$ with respect
to the functions $\boldsymbol{\mathcal{F}} = \{\mathcal{F}_2, \ldots, \mathcal{F}_{d-1}\}$, where again $\mathcal{F}_i(t) = {\rho}_{d,i} (2 + 2 t) + 2t$ for $2 \leq i \leq d-1$.
Let us denote $\mathcal{R}(\Psi) = (\mathbf{v}^{(d-1)}, \ldots, \mathbf{v}^{(1)} )$, $R_1 = \sum_{i = 1}^{d-1} \ | \mathbf{v}^{(i)} |$, and $D_1 = \sum_{i = 1}^{d-1} i \ | \mathbf{v}^{(i)} |$.

Let $\mathcal{R}^{(i)}(\Phi)$, $\Phi^{(i)}$, and $\mathcal{R}^{(i)}(\Psi)$ denote the degree $i$ forms of $\mathcal{R}^{}(\Phi)$, $\Phi$, and $\mathcal{R}^{}(\Psi)$, respectively.
From Proposition \ref{prop reg par}, we know that each $|\mathcal{R}^{(i)}(\Phi)| = |\mathbf{a}^{(i)}| \ (1 \leq i \leq d-1)$, and consequently $R_2$, is bounded by some constant dependent only on $\boldsymbol{\mathcal{F}}$, and $|\Phi^{(d-1)}|$, \ldots, $|\Phi^{(1)}|$. Thus we see that $R_2$ is bounded by a constant dependent only
on $d$.
We set
$$
M =  \rho_{d,d}( 2 + 2 R_2) + 2 R_2  + d,
$$
and note that $M$ is bounded by a constant dependent only on $d$.
By Proposition \ref{prop reg par} again,
we have that each $|\mathcal{R}^{(i)}(\Psi)| = |\mathbf{v}^{(i)}| \ (1 \leq i \leq d-1)$, and consequently $R_1$, is bounded by some constant dependent only
on $d$, $\boldsymbol{\mathcal{F}}$, $M$, and $|\Phi^{(d-1)}|$, \ldots, $|\Phi^{(1)}|$.
Thus $R_1$ is bounded by a constant dependent only on $d$ as well. \\

We define
\begin{equation}
\label{def of Ad}
A_d =  \max\{ 2 \rho_{d,d}( 2 + 2 R_1) + 4 R_1 + 2d , \ 2 \rho_{d,d}( 2 + 2 R_2) + 4 R_2 + 2d , \ \frac{5 \cdot 2^{d-1} \cdot (d-1) \cdot d!}{(\log 2)^d} + 5d\},
\end{equation}
and suppose $h^{\star}(f) \geq A_d$. We note that the third term inside the maximum function above is not required in this section, but
this lower bound on $A_d$ becomes necessary in Section \ref{section major arcs}.
With this choice of $A_d$, we have from ~(\ref{h eqn 1}) and ~(\ref{h eqn 2}) that
\begin{equation}
\label{cond on h and M}
h( f_M( \mathbf{z} ) ) \geq h - M - d \geq \rho_{d,d}( 2 + 2 R_1) + 2 R_1,
\end{equation}
and
\begin{equation}
\label{cond on h and M'}
h( g_M(\mathbf{0}, \mathbf{y}, \mathbf{z} ) ) \geq M - d \geq  \rho_{d,d}( 2 + 2 R_2) + 2 R_2.
\end{equation}

\subsection*{Step 5: Definition of the level sets $Z(\mathbf{H})$ and $Y(\mathbf{G};\mathbf{H})$}
For each $\mathbf{H} \in \mathbb{Z}^{R_1}$, we define the following set
$$
Z(\mathbf{H}) = \{ \mathbf{z} \in [0, N]^{n-M-K} \cap \mathbb{Z}^{n -M -K} : \mathcal{R}(\Psi) (\mathbf{z}) = \mathbf{H} \}.
$$
By Proposition \ref{prop reg par}, we know that each of the polynomials $\Psi^{(k)}_{t_1, \ldots, t_{j}} (\mathbf{z})$ and ${\Psi}^{(k)}_{\emptyset}(\mathbf{z})$ in ~(\ref{decomp of b 1}) can be expressed as a rational polynomial in the forms of $\mathcal{R}(\Psi)$. Let us denote
$$
\Psi^{(k)}_{t_1, \ldots, t_{j}} (\mathbf{z}) = {\hat{c} }^{(k)}_{ t_1, \ldots, t_{j}} ( \mathcal{R}(\Psi) ) \ \  \mbox{  and  } \  \
{\Psi}^{(k)}_{\emptyset} (\mathbf{z}) =  {\hat{c}}^{(k)}_{\emptyset} ( \mathcal{R}(\Psi) ),
$$
where ${\hat{c} }^{(k)}_{ t_1, \ldots, t_{j}} $ and  ${\hat{c}}^{(k)}_{\emptyset}$ are rational polynomials in $R_1$ variables.
Therefore, for any $\mathbf{z}_0 \in Z(\mathbf{H})$, we have
$$
\Psi^{(k)}_{t_1, \ldots, t_{j}}(\mathbf{z}_0) = {\hat{c} }^{(k)}_{ t_1, \ldots, t_{j}} ( \mathbf{H} ) \ \  \mbox{  and  } \  \
{\Psi}^{(k)}_{\emptyset}(\mathbf{z}_0) =  {\hat{c}}^{(k)}_{\emptyset} ( \mathbf{H} ).
$$

Since each of the forms $\widetilde{\Psi}^{(s - k)}_{{s} :i: i_1, \ldots, i_k}(\mathbf{z})$ in ~(\ref{defn of a's}) can be expressed as a rational polynomial
in the forms of $\mathcal{R}(\Psi)$, let us denote
$$
\widetilde{\Psi}^{(s - k)}_{{s} :i: i_1, \ldots, i_k}(\mathbf{z}) = \widetilde{c}^{(s - k)}_{{s} :i: i_1, \ldots, i_k}(\mathcal{R}(\Psi)),
$$
where each $\widetilde{c}^{(s - k)}_{{s} :i: i_1, \ldots, i_k}$ is a rational polynomial in $R_1$ variables.
Therefore, for each $a^{({s})}_i \in \mathcal{R}(\Phi) = ( \mathbf{a}^{(d-1)}, \ldots, \mathbf{a}^{(1)} )$,
where $1 \leq s \leq d-1$ and $1 \leq i \leq |\mathbf{a}^{(s)}|$,
we can write
\begin{equation}
\label{def of a's 2}
a^{({s})}_i(\mathbf{y}, \mathbf{z}) = \sum_{k=0}^{s} \sum_{1 \leq i_1 \leq \cdots \leq i_k \leq M} \widetilde{c}^{(s - k)}_{{s} :i: i_1, \ldots, i_k}(\mathcal{R}(\Psi)) y_{i_1} \cdots y_{i_k}.
\end{equation}
Consequently, we can define the following polynomial for each $1 \leq s \leq d-1$ and $1 \leq i \leq |\mathbf{a}^{(s)}|$,
\begin{equation}
\label{def of a's 3}
a^{({s})}_i(\mathbf{y}, Z(\mathbf{H})  ) = \sum_{k=0}^{s} \sum_{1 \leq i_1 \leq \cdots \leq i_k \leq M} \widetilde{c}^{(s - k)}_{{s} :i: i_1, \ldots, i_k}(\mathbf{H} ) y_{i_1} \cdots y_{i_k},
\end{equation}
so that given any $\mathbf{z}_0 \in Z(H)$, we have
$$
a^{({s})}_i(\mathbf{y}, \mathbf{z}_0  ) = a^{({s})}_i(\mathbf{y}, Z(\mathbf{H}) ).
$$
We also define
$$
\overline{R}(\Phi)(\mathbf{y}, Z(\mathbf{H}) ) = \{ a^{({s})}_i(\mathbf{y}, Z(\mathbf{H})  ) :
1 \leq s \leq d-1, 1 \leq i \leq |\overline{\mathbf{a}}^{(s)}|
\},
$$
which consists of $R_2$ polynomials with possible repetitions.
For each $\mathbf{G} \in \mathbb{Z}^{R_2}$, we let
$$
Y(\mathbf{G};\mathbf{H}) = \{ \mathbf{y} \in [0, N]^{ M } \cap \mathbb{Z}^{M} : \overline{\mathcal{R}}(\Phi) (\mathbf{y}, Z(\mathbf{H}) ) = \mathbf{G} \}.
$$

\subsection*{Step 6: Decomposition of $b(\mathbf{w}, \mathbf{y}, \mathbf{z})$ when $(\mathbf{y}, \mathbf{z}) \in  Y(\mathbf{G};\mathbf{H}) \times Z(\mathbf{H})$}
Recall $\Phi$ is the collection of all $\Phi^{(k)}_{i_1, \ldots, i_{j}} ( \mathbf{y}, \mathbf{z} )$ in ~(\ref{decomp of b 1}), and
that each $\Phi^{(k)}_{i_1, \ldots, i_{j}} ( \mathbf{y}, \mathbf{z} )$ can be expressed as
a rational polynomial in the forms of $\mathcal{R}(\Phi)$.
Thus, it follows from this fact and ~(\ref{def of a's 3}) that each $\Phi^{(k)}_{i_1, \ldots, i_{j}} ( \mathbf{y}, \mathbf{z} )$
is constant on $(\mathbf{y}, \mathbf{z}) \in Y(\mathbf{G};\mathbf{H}) \times Z(\mathbf{H})$, and we denote this constant value by
${c}^{(k)}_{i_1, \ldots, i_{j}} ( \mathbf{G}, \mathbf{H} )$.
Therefore, for any choice of $\mathbf{z} \in Z(\mathbf{H})$ and $ \mathbf{y} \in  Y(\mathbf{G};\mathbf{H})$, the polynomial ${b}(\mathbf{x})$ takes the following
shape
\begin{eqnarray}
&&{b}( \mathbf{w}, \mathbf{y}, \mathbf{z} )
\label{decom of b with coeff}
\\
&=&
{b}( \mathbf{w}, \mathbf{0}, \mathbf{0})
+ \sum_{j=1}^{d-1}  \ \sum_{1 \leq i_1 \leq \cdots \leq i_j \leq K}
\left( \sum_{k=1}^{d-j} {c}^{(k)}_{i_1, \ldots, i_{j}} ( \mathbf{G}, \mathbf{H} )  \right) w_{i_1} \cdots w_{i_j}
\notag
\\
&+&
\sum_{j=1}^{d-1} \ \sum_{1 \leq t_1 \leq \cdots \leq t_{j} \leq M}
\left( \sum_{k=0}^{d-j} {\hat{c} }^{(k)}_{ t_1, \ldots, t_{j}} ( \mathbf{H} )  \right) y_{t_1} \cdots y_{t_{j}}
+
\left( \sum_{k=1}^{d} {\hat{c}}^{(k)}_{\emptyset} ( \mathbf{H} )  \right)
+ g_M( \mathbf{0}, \mathbf{y}, \mathbf{0} )
\notag
\\
&+& b_M(\mathbf{z}) - b(\mathbf{0}, \mathbf{0}, \mathbf{0}).
\notag
\end{eqnarray}

We label
\begin{eqnarray}
\mathfrak{C}_0 (\mathbf{w}, \mathbf{G}, \mathbf{H}  ) &=& {b}(\mathbf{w}, \mathbf{0}, \mathbf{0} )
+ \sum_{j=1}^{d-1}  \ \sum_{1 \leq i_1 \leq \cdots \leq i_j \leq K}
\left( \sum_{k=1}^{d-j} {c}^{(k)}_{i_1, \ldots, i_{j}} ( \mathbf{G}, \mathbf{H} )  \right) w_{i_1} \cdots w_{i_j},
\notag
\end{eqnarray}
and
$$
\mathfrak{C}_1 (\mathbf{y}, \mathbf{H} ) = \sum_{j=1}^{d-1} \ \sum_{1 \leq t_1 \leq \cdots \leq t_{j} \leq M}
\left( \sum_{k=0}^{d-j} {\hat{c}}^{(k)}_{ t_1, \ldots, t_{j}} ( \mathbf{H} )  \right) y_{t_1} \cdots y_{t_{j}}
+
\left( \sum_{k=1}^{d} {\hat{c}}^{(k)}_{\emptyset} ( \mathbf{H} )  \right)
+ g_M( \mathbf{0}, \mathbf{y}, \mathbf{0} ),
$$
so that for $\mathbf{z} \in Z(\mathbf{H})$ and $ \mathbf{y} \in  Y(\mathbf{G};\mathbf{H})$, we have
\begin{equation}
\label{decom of b last}
{b}(\mathbf{w}, \mathbf{y}, \mathbf{z}  ) =  \mathfrak{C}_0 (\mathbf{w}, \mathbf{G}, \mathbf{H}  ) + \mathfrak{C}_1 (\mathbf{y}, \mathbf{H} ) +
b_M(\mathbf{z}) - b(\mathbf{0}, \mathbf{0}, \mathbf{0}).
\end{equation}

\subsection*{Proof of Proposition \ref{prop minor arc bound}}
We are now in position to prove Proposition \ref{prop minor arc bound}.
\begin{proof} 
Using the notations above we define the following three exponential sums,
$$
S_0( {\alpha}, \mathbf{G}, \mathbf{H} ) = \sum_{ \mathbf{w} \in [0,N]^K \cap \mathbb{Z}^K } \Lambda (\mathbf{w})
\ e(  {\alpha} \cdot \mathfrak{C}_0 (\mathbf{w}, \mathbf{G}, \mathbf{H}  )  ),
$$
$$
S_1( {\alpha}, \mathbf{G}, \mathbf{H}) = \sum_{\mathbf{y} \in Y(\mathbf{G};\mathbf{H})} \Lambda (\mathbf{y}) \  e( {\alpha} \cdot \mathfrak{C}_1 (\mathbf{y}, \mathbf{H}  )  ),
$$
and
$$
S_2( {\alpha}, \mathbf{H} ) = \sum_{\mathbf{z} \in Z(\mathbf{H}) } \Lambda (\mathbf{z}) \ e(  {\alpha} \cdot   b_M(\mathbf{z}) - {\alpha} \cdot b(\mathbf{0}, \mathbf{0}, \mathbf{0}) ).
$$

Let
\begin{eqnarray}
\mathcal{L}_1(N) 
= \{ \mathbf{H} \in \mathbb{Z}^{R_1} : Z(\mathbf{H}) \not = \emptyset \},
\notag
\end{eqnarray}
and for each $\mathbf{H} \in \mathcal{L}_1(N)$,  let
\begin{eqnarray}
\mathcal{L}_2(N ; \mathbf{H}) 
= \{ \mathbf{G} \in \mathbb{Z}^{R_2} : Y(\mathbf{G}, \mathbf{H} ) \not = \emptyset \}.
\notag
\end{eqnarray}
It then follows that
\begin{equation}
\label{ineq 1}
|\mathcal{L}_1(N)| \ll N^{D_1} \  \mbox{ and } \  |\mathcal{L}_2(N; \mathbf{H})| \ll N^{D_2},
\end{equation}
where the implicit constant in the second inequality is independent of $\mathbf{H}$.
In order to prove the first inequality, let $C_0$ be the largest absolute value of all coefficients of the polynomials in $\mathcal{R}(\Psi)$.
Also let $M_0$ be the largest number of monomials with non-zero coefficients in any of the polynomials in $\mathcal{R}(\Psi)$.
Then we have
$$
|\mathcal{L}_1(N)| \leq  (2 C_0 \cdot M_0)^{R_1} \cdot  (N+1)^{D_1}.
$$
To see the second inequality, we let $C'_0$ be the largest absolute value of all coefficients of the polynomials
$a_i^{(s)} (\mathbf{y}, \mathbf{z})$ in $\overline{\mathcal{R}}(\Phi)$, 
and let $M'_0$ be the largest number of monomials with non-zero coefficients in any of these polynomials.
Then we see that the number of values taken by $a_i^{(s)} (\mathbf{y}, \mathbf{z})$ as $(\mathbf{y}, \mathbf{z})$ varies in $[0, N]^{n-K}$ is
$$
\leq (2 C_0' \cdot M_0') \cdot (N+1)^{s}.
$$
Therefore, we have
\begin{eqnarray}
\mathcal{L}_2(N ; \mathbf{H})
&=& \{ \mathbf{G} \in \mathbb{Z}^{R_2} : Y(\mathbf{G}, \mathbf{H} ) \not = \emptyset \}
\notag
\\
&=& \{ \mathbf{G} \in \mathbb{Z}^{R_2} : \exists  \mathbf{y} \in [0, N]^{ M } \cap \mathbb{Z}^{M}, \overline{\mathcal{R}}(\Phi) (\mathbf{y}, Z(\mathbf{H}) ) = \mathbf{G}  \}
\notag
\\
&\subseteq &
\{ \mathbf{G} \in \mathbb{Z}^{R_2} : \exists (\mathbf{y}, \mathbf{z}) \in [0, N]^{ n- K  } \cap \mathbb{Z}^{n - K},  \overline{\mathcal{R}}(\Phi) (\mathbf{y}, \mathbf{z} ) = \mathbf{G} \},
\notag
\end{eqnarray}
and the cardinality of the last set is
$$
\leq (2 C_0' \cdot M_0')^{R_2} \cdot (N+1)^{D_2}.
$$

By the Cauchy-Schwarz inequality and (\ref{ineq 1}), we obtain
\begin{eqnarray}
\label{minor arc ineq 1}
&& \Big{|} \int_{\mathfrak{m}(C)} T({b}; {\alpha} ) \ {d} {\alpha}  \Big{|}^2
\\
&\leq&
\Big{|}
\sum_{\mathbf{H} \in \mathcal{L}_1(N)} \sum_{\mathbf{G} \in \mathcal{L}_2(N ; \mathbf{H})} \int_{\mathfrak{m}(C)} \
\sum_{\substack{ \mathbf{w} \in [0,N]^K \cap \mathbb{Z}^K \\ \mathbf{z} \in Z(\mathbf{H}) \\  \mathbf{y} \in Y(\mathbf{G}; \mathbf{H}) }}
\Lambda(\mathbf{w}) \Lambda(\mathbf{y}) \Lambda(\mathbf{z}) \cdot
\notag
\\
&&\phantom{12345612345}
e( {\alpha} \cdot (\mathfrak{C}_0 (\mathbf{w}, \mathbf{G}, \mathbf{H}  ) + \mathfrak{C}_1 (\mathbf{y}, \mathbf{H} ) + {b}_M( \mathbf{z} ) - b(\mathbf{0}, \mathbf{0}, \mathbf{0})  \  ))
\  {d} {\alpha}  \Big{|}^2
\notag
\\
&\ll&
N^{D_1+ D_2}
\sum_{\mathbf{H} \in \mathcal{L}_1(N)} \sum_{\mathbf{G} \in \mathcal{L}_2(N ; \mathbf{H})} \Big{|} \int_{\mathfrak{m}(C)}
   S_0 ( {\alpha}, \mathbf{G}, \mathbf{H} )
S_1( {\alpha}, \mathbf{G}, \mathbf{H} )S_2( {\alpha}, \mathbf{H} )
 \  {d}  {\alpha} \Big{|}^2
\notag
\\
&\ll&
N^{D_1+D_2} \  \left( \sup_{\substack {\mathbf{H} \in \mathcal{L}_1(N) \\ \mathbf{G} \in \mathcal{L}_2(N ; \mathbf{H}) }}  \sup_{\alpha \in \mathfrak{m}(C) } | S_0 (\alpha, \mathbf{G}, \mathbf{H} ) |^2 \right) \cdot
\notag
\\
&&\phantom{1223344556677889123456}  \sum_{\mathbf{H} \in \mathcal{L}_1(N)} \sum_{\mathbf{G} \in \mathcal{L}_2(N ; \mathbf{H})}   \|S_1(\cdot, \mathbf{G}, \mathbf{H} ) \|_2^2 \ \| S_2(\cdot, \mathbf{H} ) \|_2^2,
\notag
\end{eqnarray}
where $\| \cdot \|_2$ denotes the $L^2$-norm on $[0,1]$.
By the orthogonality relation, it follows that
\begin{eqnarray}
\notag
\|S_1(\cdot, \mathbf{G}, \mathbf{H} ) \|_2^2 \ \| S_2(\cdot, \mathbf{H} ) \|_2^2
&\leq&
(\log N)^{2n - 2K} \mathcal{N}_1 (\mathbf{G}; \mathbf{H}) \mathcal{N}_2 (\mathbf{H}),
\notag
\end{eqnarray}
where
$$\mathcal{N}_1 (\mathbf{G}; \mathbf{H}) =  | \{ (\mathbf{y}, \mathbf{y}' )  \in Y(\mathbf{G}; \mathbf{H}) \times Y(\mathbf{G}; \mathbf{H}): \mathfrak{C}_1 (\mathbf{y}, \mathbf{H}  ) = \mathfrak{C}_1 (\mathbf{y}', \mathbf{H}  )   \} |,$$
and
$$
\mathcal{N}_2 (\mathbf{H}) =  |\{ (\mathbf{z}, \mathbf{z}' )  \in Z(\mathbf{H}) \times Z(\mathbf{H}) : {b}_M( \mathbf{z} ) = {b}_M( \mathbf{z}' ) \}|.
$$

With these notations, we may further bound ~(\ref{minor arc ineq 1}) as follows
\begin{eqnarray}
\label{minor arc ineq 2}
\Big{|} \int_{\mathfrak{m}(C)} T({b}; {\alpha} ) \ {d} {\alpha}  \Big{|}^2
\ll
(\log N)^{2n - 2K}  N^{D_1+D_2} \ \left( \sup_{\substack {\mathbf{H} \in \mathcal{L}_1(N) \\ \mathbf{G} \in \mathcal{L}_2(N ; \mathbf{H}) }}  \sup_{\alpha \in \mathfrak{m}(C) } | S_0 (\alpha, \mathbf{G}, \mathbf{H} ) |^2 \right)
\ \mathcal{W},
\notag
\end{eqnarray}
where
$$
\mathcal{W} = \sum_{\mathbf{H} \in \mathcal{L}_1(N)} \sum_{\mathbf{G} \in \mathcal{L}_2(N; \mathbf{H})} \mathcal{N}_1 (\mathbf{G}; \mathbf{H})  \mathcal{N}_2 (\mathbf{H}).
$$
We can express $\mathcal{W}$ as the number of solutions $\mathbf{y}, \mathbf{y}'  \in [0,N]^M \cap \mathbb{Z}^M$
and $\mathbf{z}, \mathbf{z}' \in [0,N]^{n-M-K} \cap \mathbb{Z}^{n-M-K} $ of the system
\begin{eqnarray}
{\mathcal{R}}(\Psi) ( \mathbf{z}) &=& {\mathcal{R}}(\Psi) (\mathbf{z}') = \mathbf{H}
\label{first system}
\\
\overline{\mathcal{R}}(\Phi) (\mathbf{y},  Z(\mathbf{H})) &=& \overline{\mathcal{R}}(\Phi) (\mathbf{y}', Z(\mathbf{H}))= \mathbf{G}
\notag
\\
\mathfrak{C}_1 (\mathbf{y}, \mathbf{H}  )  &=&  \mathfrak{C}_1 (\mathbf{y}', \mathbf{H}  )
\notag
\\
{b}_M( \mathbf{z} ) &=& {b}_M( \mathbf{z}' )
\notag
\end{eqnarray}
for any $\mathbf{H}\in \mathcal{L}_1(N)$ and $\mathbf{G}\in \mathcal{L}_2(N; \mathbf{H})$.
We know that the system $\overline{\mathcal{R}}(\Phi) (\mathbf{y},  Z(\mathbf{H}))$ is identical to $\overline{\mathcal{R}}(\Phi) (\mathbf{y},  \mathbf{z}_0)$ for any choice of $\mathbf{z}_0 \in Z(\mathbf{H})$ and any $\mathbf{y} \in [0,N]^M \cap \mathbb{Z}^M$. Similarly, we know that the polynomial
$\mathfrak{C}_1 (\mathbf{y}, \mathbf{H})$ is identical to $b(\mathbf{0}, \mathbf{y}, \mathbf{z}_0) - b_M(\mathbf{z}_0)$ for
any choice of $\mathbf{z}_0 \in Z(\mathbf{H})$.
Therefore, since ${\mathcal{R}}(\Psi) (\mathbf{z}) = \mathbf{H}$
implies that $z \in Z(\mathbf{H})$, we can rearrange the system ~(\ref{first system}) and
deduce that $\mathcal{W}$ is the number of solutions $\mathbf{y}, \mathbf{y}'  \in [0,N]^M \cap \mathbb{Z}^M$
and $\mathbf{z}, \mathbf{z}' \in [0,N]^{n-M-K} \cap \mathbb{Z}^{n-M-K} $ of the following system
\begin{eqnarray}
\label{second system}
{\mathcal{R}}(\Psi) ( \mathbf{z}) &=& {\mathcal{R}}(\Psi) (\mathbf{z}')
\\
\overline{\mathcal{R}}(\Phi) (\mathbf{y},  \mathbf{z} ) &=& \overline{\mathcal{R}}(\Phi) (\mathbf{y}', \mathbf{z} )
\notag
\\
{b}(\mathbf{0}, \mathbf{y}, \mathbf{z}) - {b}_M (\mathbf{z} )  &=&  {b}(\mathbf{0}, \mathbf{y}', \mathbf{z}) - {b}_M (\mathbf{z} )
\notag
\\
{b}_M( \mathbf{z} ) &=& {b}_M( \mathbf{z}' ).
\notag
\end{eqnarray}

Our result follows from the following two claims.

Claim 1: Given any $c>0$, there exists $C > 0$ such that the following bound holds,
$$
\sup_{\substack {\mathbf{H} \in \mathcal{L}_1(N) \\ \mathbf{G} \in \mathcal{L}_2(N ; \mathbf{H}) }}  \sup_{\alpha \in \mathfrak{m}(C) } | S_0 (\alpha, \mathbf{G}, \mathbf{H} ) |
\ll \frac{N^K}{(\log N)^c}.
$$

Claim 2: We have the following bound on $\mathcal{W}$,
$$
\mathcal{W} \ll N^{2n - 2K - 2 d - D_1 - D_2}.
$$

By substituting the bounds from the above two claims into ~(\ref{minor arc ineq 2}), we obtain for any $c>0$
there exists $C>0$ such that
$$
\int_{\mathfrak{m}(C)} T({b}; {\alpha} ) \ {d} {\alpha}
\ll
\frac{N^{n - d} }{(\log N)^{c}},
$$
and this completes the proof of our proposition.
Therefore, we only need to establish Claims $1$ and $2$.
Claim $1$ is obtained via Weyl differencing. Since the set up for our Claim 1 is the same as that
of \cite{CM}, we omit the proof of Claim 1 and refer the reader to \cite[pp. 725]{CM}.

We now present the proof of Claim 2. From ~(\ref{second system}), we can write
$$
\mathcal{W} = \sum_{\mathbf{z} \in  [0,N]^{n-M-K} \cap \mathbb{Z}^{n-M-K}} T_1(\mathbf{z}) \cdot T_2(\mathbf{z}),
$$
where $T_1(\mathbf{z})$ is the number of solutions $\mathbf{y}, \mathbf{y}' \in [0,N]^{M} \cap \mathbb{Z}^{M}$
to the system
\begin{eqnarray}
{b}( \mathbf{0}, \mathbf{y}, \mathbf{z} ) &=&  {b}( \mathbf{0}, \mathbf{y}', \mathbf{z} )
\notag
\\
\overline{\mathcal{R}}(\Phi) (\mathbf{y}, \mathbf{z}) &=& \overline{\mathcal{R}}(\Phi) (\mathbf{y}', \mathbf{z}),
\notag
\end{eqnarray}
and $T_2(\mathbf{z})$ is the number of solutions $\mathbf{z}' \in  [0,N]^{n-M-K} \cap \mathbb{Z}^{n-M-K}$
to the system
\begin{eqnarray}
{b}_M( \mathbf{z} ) &=& {b}_M( \mathbf{z}' )
\notag
\\
{\mathcal{R}}(\Psi) ( \mathbf{z}) &=& {\mathcal{R}}(\Psi) (\mathbf{z}').
\notag
\end{eqnarray}
Define $\mathcal{W}_i =\sum_{\mathbf{z}} T_i (\mathbf{z})^2 \ (i=1, 2)$ so that
we have $\mathcal{W}^2 \leq \mathcal{W}_1 \mathcal{W}_2$ by the Cauchy-Schwarz inequality.
We first estimate $\mathcal{W}_1$, which
we can deduce to be the number of solutions $\mathbf{y}, \mathbf{y}', \mathbf{u}, \mathbf{u}' \in [0,N]^M \cap \mathbb{Z}^M$
and $\mathbf{z} \in [0,N]^{n-M-K} \cap \mathbb{Z}^{n-M-K}$ satisfying the equations
\begin{eqnarray}
\label{system 1 in claim 2}
{b}( \mathbf{0}, \mathbf{y}, \mathbf{z} )  -  {b}( \mathbf{0}, \mathbf{y}', \mathbf{z} ) &=& 0
\\
{b}( \mathbf{0}, \mathbf{u}, \mathbf{z} ) - {b}( \mathbf{0}, \mathbf{u}', \mathbf{z} )  &=&  0
\notag
\\
\overline{\mathcal{R}}(\Phi) (\mathbf{y}, \mathbf{z}) - \overline{\mathcal{R}}(\Phi) (\mathbf{y}', \mathbf{z}) &=&  0
\notag
\\
\overline{\mathcal{R}}(\Phi) (\mathbf{u}, \mathbf{z}) - \overline{\mathcal{R}}(\Phi) (\mathbf{u}', \mathbf{z}) &=&  0.
\notag
\end{eqnarray}
We consider the $h$-invariant of the system of forms on the left hand side of ~(\ref{system 1 in claim 2}), and show that
it is a regular system. The first two equations of ~(\ref{system 1 in claim 2}) are the degree $d$ polynomials of the system, and let $h_d$ be the $h$-invariant of these two polynomials.
Suppose for some $\lambda, \mu \in \mathbb{Q}$, not both $0$, we have
\begin{eqnarray}
\lambda \cdot ({f}( \mathbf{0}, \mathbf{y}, \mathbf{z} )  -  {f}( \mathbf{0}, \mathbf{y}', \mathbf{z} ) )
+ \mu \cdot ({f}( \mathbf{0}, \mathbf{u}, \mathbf{z} ) - {f}( \mathbf{0}, \mathbf{u}', \mathbf{z} ) )
= \sum_{j=1}^{h_d} U_j \cdot V_j,
\notag
\end{eqnarray}
where $U_j = U_j( \mathbf{y}, \mathbf{y}', \mathbf{u}, \mathbf{u}', \mathbf{z} )$ and $V_j =
V_j( \mathbf{y}, \mathbf{y}', \mathbf{u}, \mathbf{u}', \mathbf{z} )$ are rational forms of positive degree $(1 \leq j \leq h_d)$.
Without loss of generality, suppose $\lambda \not = 0$.
Let $\boldsymbol{\ell} = (- \ell_1 |_{\mathbf{w} = \mathbf{0}}, \ldots, - \ell_M |_{\mathbf{w} = \mathbf{0}} )$. If we set
$\mathbf{u} = \mathbf{u}' = \mathbf{y}' = \boldsymbol{\ell}$,
then the above equation becomes
\begin{eqnarray}
{g_M}( \mathbf{0}, \mathbf{y}, \mathbf{z} ) = {f}( \mathbf{0}, \mathbf{y}, \mathbf{z} )  -  {f}_M( \mathbf{z} )
= \frac{1}{\lambda} \sum_{j=1}^{h_d}  U_j (\mathbf{y}, \boldsymbol{\ell},\boldsymbol{\ell}, \boldsymbol{\ell}, \mathbf{z}) \cdot V_j (\mathbf{y}, \boldsymbol{\ell}, \boldsymbol{\ell},\boldsymbol{\ell}, \mathbf{z}).
\notag
\end{eqnarray}
Therefore, we obtain from ~(\ref{cond on h and M'})
\begin{eqnarray}
h_d \geq  h ({g_M}( \mathbf{0}, \mathbf{y}, \mathbf{z} ) )
\notag
\geq \rho_{d,d}( 2 + 2 R_2) + 2 R_2
\notag
\geq \rho_{d,d}( 2 + 2 R_2 - 2 | { \overline{\mathbf{a} }}^{(1)}|) + 2 | { \overline{\mathbf{a} }}^{(1)}|.
\notag
\end{eqnarray}

For each $1 \leq i \leq d-1$, denote by
$$
\overline{\mathcal{R}}(\Phi)^{(i)} (\mathbf{y}, \mathbf{z}) - \overline{\mathcal{R}}(\Phi)^{(i)} (\mathbf{y}', \mathbf{z})
=
\{ {a}_j^{(i)}(\mathbf{y}, \mathbf{z}) - {a}_j^{(i)}(\mathbf{y}', \mathbf{z}) : 1 \leq j \leq |\overline{\mathbf{a}}^{(i)}| \},
$$
the system of degree $i$ polynomials of $\overline{\mathcal{R}}(\Phi) (\mathbf{y}, \mathbf{z}) - \overline{\mathcal{R}}(\Phi) (\mathbf{y}', \mathbf{z})$.
We also define
$$
\overline{\mathcal{R}}(\Phi)^{(i)} (\mathbf{u}, \mathbf{z}) - \overline{\mathcal{R}}(\Phi)^{(i)} (\mathbf{u}', \mathbf{z})
$$
in a similar manner. We apply Lemma \ref{Lemma 2 in CM} to estimate the $h$-invariant of the degree $i$ forms of the system ~(\ref{system 1 in claim 2}) for each $2 \leq i \leq d-1$,
\begin{eqnarray}
&&
h \left( \overline{\mathcal{R}}(\Phi)^{(i)} (\mathbf{y}, \mathbf{z}) - \overline{\mathcal{R}}(\Phi)^{(i)} (\mathbf{y}', \mathbf{z}),  \overline{\mathcal{R}}(\Phi)^{(i)} (\mathbf{u}, \mathbf{z}) - \overline{\mathcal{R}}(\Phi)^{(i)} (\mathbf{u}', \mathbf{z} ) \right)
\notag
\\
&\geq &h \left( \overline{\mathcal{R}}(\Phi)^{(i)} (\mathbf{y}, \mathbf{z}) - \overline{\mathcal{R}}(\Phi)^{(i)} (\mathbf{y}', \mathbf{z}),  \overline{\mathcal{R}}(\Phi)^{(i)} (\mathbf{u}, \mathbf{z}) - \overline{\mathcal{R}}(\Phi)^{(i)} (\mathbf{u}', \mathbf{z} ) ; \mathbf{z} \right)
\notag
\\
&=&
h \left( \overline{\mathcal{R}}(\Phi)^{(i)} (\mathbf{y}, \mathbf{z}) - \overline{\mathcal{R}}(\Phi)^{(i)} (\mathbf{y}', \mathbf{z}) ; \mathbf{z}  \right)
\notag
\\
&\geq&
h \left( \overline{\mathcal{R}}(\Phi)^{(i)} (\mathbf{y}, \mathbf{z}) , \overline{\mathcal{R}}(\Phi)^{(i)} (\mathbf{y}', \mathbf{z}) ; \mathbf{z}  \right)
\notag
\\
&\geq&
h \left( \overline{\mathcal{R}}(\Phi)^{(i)} (\mathbf{y}, \mathbf{z}) ; \mathbf{z}  \right)
\notag
\\
&\geq&
\rho_{d,i}( 2 + 2 R_2 ) + 2 R_2
\notag
\\
&\geq&
\rho_{d,i}( 2 + 2 R_2 - 2| \overline{{ \mathbf{a} }}^{(1)}| ) + 2 | \overline{{ \mathbf{a} }}^{(1)}|.
\notag
\end{eqnarray}

We also have to show that the linear forms of the system ~(\ref{system 1 in claim 2}) are linearly independent over $\mathbb{Q}$.
Recall the linear forms of $\overline{\mathcal{R}}(\Phi)^{(1)} (\mathbf{y}, \mathbf{z})$ are linearly independent over $\mathbb{Q}$,
and do not include any linear forms that depend only on the $\mathbf{z}$ variables, and similarly for $\overline{\mathcal{R}}(\Phi)^{(1)} (\mathbf{y}', \mathbf{z})$, $\overline{\mathcal{R}}(\Phi)^{(1)} (\mathbf{u}, \mathbf{z})$, and $\overline{\mathcal{R}}(\Phi)^{(1)} (\mathbf{u}', \mathbf{z})$.
We leave it as a basic exercise for the reader to verify that the linear forms of
$$
\overline{\mathcal{R}}(\Phi)^{(1)} (\mathbf{y}, \mathbf{z}) - \overline{\mathcal{R}}(\Phi)^{(1)} (\mathbf{y}', \mathbf{z}) \
\bigcup
\ \overline{\mathcal{R}}(\Phi)^{(1)} (\mathbf{u}, \mathbf{z}) - \overline{\mathcal{R}}(\Phi)^{(1)} (\mathbf{u}', \mathbf{z})
$$
are linearly independent over $\mathbb{Q}$.

Therefore, it follows from Corollary \ref{cor Schmidt} that
$$
\mathcal{W}_1 \ll N^{n + 3M - K - (2 d + 2 D_2)}.
$$

We now estimate $\mathcal{W}_2$,
which we can deduce to be the number of solutions $\mathbf{z}, \mathbf{z}', \mathbf{z}'' \in [0,N]^{n-M-K } \cap \mathbb{Z}^{n-M-K}$
satisfying the equations
\begin{eqnarray}
\label{system 2 in claim 2}
{b}_M( \mathbf{z} ) - {b}_M( \mathbf{z}' ) &=& 0
\\
{b}_M( \mathbf{z} ) - {b}_M(  \mathbf{z}'' ) &=& 0
\notag
\\
{\mathcal{R}}(\Psi) ( \mathbf{z}) - {\mathcal{R}}(\Psi) (\mathbf{z}') &=& 0
\notag
\\
{\mathcal{R}}(\Psi) ( \mathbf{z}) - {\mathcal{R}}(\Psi) (\mathbf{z}'') &=& 0.
\notag
\end{eqnarray}
We consider the $h$-invariant of the system of forms on the left hand side of ~(\ref{system 2 in claim 2}), and show that
it is a regular system. The first two equations of ~(\ref{system 2 in claim 2})
are the degree $d$ polynomials of the system, and let $h_d$ be the $h$-invariant of these two polynomials.
Suppose for some $\lambda, \mu \in \mathbb{Q}$, not both $0$, we have
\begin{eqnarray}
\label{eqn f_M U V}
\lambda \cdot ( {f}_M ( \mathbf{z} ) - {f}_M (  \mathbf{z}' ) )
+ \mu \cdot ( {f}_M ( \mathbf{z} ) - {f}_M ( \mathbf{z}'' ) )
= \sum_{j=1}^{h_d} U_j \cdot V_j,
\notag
\end{eqnarray}
where $U_j = U_j( \mathbf{z}, \mathbf{z}', \mathbf{z}'' )$ and $V_j =
V_j( \mathbf{z}, \mathbf{z}', \mathbf{z}'' )$ are rational forms of positive degree $(1 \leq j \leq h_d)$.
We consider two cases, $(\lambda + \mu ) \not = 0$ and $(\lambda + \mu ) = 0$.
Suppose $(\lambda + \mu ) \not = 0$. If we set $\mathbf{z}' = \mathbf{z}'' = \mathbf{0}$, then the above equation becomes
\begin{eqnarray}
(\lambda + \mu ) \cdot  {f}_M( \mathbf{z} )
=  \sum_{j=1}^{h_d}  U_j (\mathbf{z}, \mathbf{0}, \mathbf{0} ) \cdot V_j (\mathbf{z}, \mathbf{0}, \mathbf{0} ).
\notag
\end{eqnarray}
Thus we obtain $h_d \geq h( {f}_M( \mathbf{z} ) ).$
On the other hand, suppose $(\lambda + \mu ) = 0$, then the above equation ~(\ref{eqn f_M U V}) simplifies to
$$
{f}_M(  \mathbf{z}' ) - {f}_M(  \mathbf{z}'' )
= \frac{-1}{\lambda} \sum_{j=1}^{h_d} U_j \cdot V_j.
$$
From this equation, 
we substitute $\mathbf{z}'' = \mathbf{0}$ to obtain $h_d \geq h( {f}_M(  \mathbf{z}' ) ).$
Therefore, in either case we obtain from ~(\ref{cond on h and M}) that
$$
h_d \geq h( {f}_M( \mathbf{z} ) ) \geq \rho_{d,d}( 2 + 2  R_1) + 2R_1 \geq \rho_{d,d}( 2 + 2  R_1 - 2 | { \mathbf{v} }^{(1)}|) + 2 | { \mathbf{v} }^{(1)}|.
$$
Recall we defined ${\mathcal{R}} (\Psi) = (\mathbf{v}^{(d-1)}, \ldots, \mathbf{v}^{(1)})$, where $\mathbf{v}^{(i)} = \mathcal{R}^{(i)}(\Psi)$ are the degree $i$ forms of
${\mathcal{R}} (\Psi) \ (1 \leq i \leq d-1)$. Take $2 \leq i \leq d-1$. Let $m_i = | \mathbf{v}^{(i)} |$, and we label the forms of $\mathbf{v}^{(i)}$ to be $v^{(i)}_1, \ldots,  v^{(i)}_{m_i}$. 
Let $h_i$ be the $h$-invariant of the degree $i$ forms of the system ~(\ref{system 2 in claim 2}).
Then for some $\boldsymbol{\lambda}, \boldsymbol{\mu} \in \mathbb{Q}^{m_i}$, not both $\mathbf{0}$, we have
\begin{equation}
\label{eqn f_M U V 1}
\sum_{j=1}^{m_i} \lambda_j \cdot ( v^{(i)}_j (\mathbf{z}) - v^{(i)}_j (\mathbf{z}' ) ) + \sum_{j=1}^{m_i} \mu_j \cdot ( v^{(i)}_j (\mathbf{z}) - v^{(i)}_j (\mathbf{z}'' ) )= \sum_{t=1}^{h_i} U_t \cdot V_t,
\end{equation}
where $U_t = U_t( \mathbf{z}, \mathbf{z}', \mathbf{z}'')$ and $V_t =
V_t( \mathbf{z}, \mathbf{z}', \mathbf{z}'')$ are forms of positive degree $(1 \leq t \leq h_i)$.
We consider two cases, $(\boldsymbol{\lambda} + \boldsymbol{\mu}) \not = \mathbf{0}$ and
$(\boldsymbol{\lambda} + \boldsymbol{\mu}) = \mathbf{0}$.

Suppose $(\boldsymbol{\lambda} + \boldsymbol{\mu}) \not = \mathbf{0}$. In this case, we set $\mathbf{z}' = \mathbf{z}'' = \mathbf{0}$, and equation ~(\ref{eqn f_M U V 1}) simplifies to
$$
\sum_{j=1}^{m_i} ( \lambda_j + \mu_j ) \cdot  v^{(i)}_j (\mathbf{z}) = \sum_{t=1}^{h_i} U_t( \mathbf{z}, \mathbf{0}, \mathbf{0}) \cdot  V_t ( \mathbf{z}, \mathbf{0}, \mathbf{0}).
$$
Therefore, it follows that
$$
h_i \geq h( { \mathbf{v} }^{(i)} ) \geq \rho_{d,i}( 2 + 2 R_1) + 2 R_1 \geq \rho_{d,i}( 2 + 2 R_1 - 2| { \mathbf{v} }^{(i)}|) + 2| { \mathbf{v} }^{(1)}|.
$$
On the other hand, suppose $(\boldsymbol{\lambda} + \boldsymbol{\mu}) = \mathbf{0}$.
Then equation ~(\ref{eqn f_M U V 1}) simplifies to
$$
\sum_{j=1}^{m_i} - \lambda_j \cdot (   v^{(i)}_j (\mathbf{z}') -  v^{(i)}_j (\mathbf{z}'' ) )= \sum_{t=1}^{h_i} U_t \cdot V_t.
$$
From this equation, 
we substitute $\mathbf{z}'' = \mathbf{0}$ to obtain
$$
h_i \geq h( { \mathbf{v} }^{(i)} ) \geq \rho_{d,i}( 2 + 2 R_1) + 2 R_1 \geq \rho_{d,i}( 2 + 2 R_1 - 2| { \mathbf{v} }^{(i)}|) + 2| { \mathbf{v} }^{(1)}|.
$$

We also have to show that the linear forms of the system ~(\ref{system 2 in claim 2}),
\begin{equation}
\label{system of linear forms2}
\{ \mathbf{v}^{(1)}(\mathbf{z}) - \mathbf{v}^{(1)}(\mathbf{z}') \} \cup \{ \mathbf{v}^{(1)}(\mathbf{z}) - \mathbf{v}^{(1)}(\mathbf{z}'') \},
\end{equation}
are linearly independent over $\mathbb{Q}$.
Recall the linear forms of $\mathbf{v}^{(1)}(\mathbf{z})$ are linearly independent over $\mathbb{Q}$.
The linear independence over $\mathbb{Q}$ of the system of linear forms ~(\ref{system of linear forms2}) follows from this fact, and we leave the verification
as a basic exercise for the reader.

Therefore, we obtain by Corollary \ref{cor Schmidt} that
$$
\mathcal{W}_2 \ll N^{3(n -M - K) - (2 d + 2 D_1)}.
$$

Combining the bounds for $\mathcal{W}_1$ and $\mathcal{W}_2$ together, we obtain
$$
\mathcal{W} \leq \mathcal{W}_1^{1/2} \mathcal{W}_2^{1/2} \ll N^{2n -2 K - (2d + D_1 + D_2)},
$$
which proves Claim 2.
\end{proof}

\section{Hardy-Littlewood Circle Method: Major Arcs}
\label{section major arcs}
Recall $f(\mathbf{x})$ is the degree $d$ portion of the degree $d$ polynomial $b(\mathbf{x}) \in \mathbb{Z}[x_1, \ldots, x_n]$.
In this section we assume that $f(\mathbf{x})$ satisfies $h(f) > A_d$, where $A_d$ is defined in ~(\ref{def of Ad}).
We define $g_d(f)$ as in ~(\ref{def gd}) with $\mathbf{f} = \{ f \}$ and $r_d = 1$.
It then follows from ~(\ref{h and g}) 
that
$$
A_d < h(f) \leq (\log 2)^{-d} \cdot d! \cdot g_d(f).
$$
From this bound and our choice of $A_d$ in ~(\ref{def of Ad}), we have
\begin{equation}
\label{g and h}
\frac{ 2^{d-1}}{g_d(f)} < \frac{d!  2^{d-1} }{(\log 2)^d A_d} < \frac{d! 2^{d-1} }{(\log 2)^d (A_d - 5d) } \leq \frac{1}{5(d-1)}.
\end{equation}

We take $\Omega$ to be
$$
4 \ < \ \Omega  \ < \ 5 \ \leq  \ \frac{(A_d - 5 d) \cdot  (\log 2)^{d}}{ 2^{d-1} (d-1) d! } \ \leq \  \frac{g_d(f)}{ 2^{d-1} (d-1)}.
$$
Therefore, with this choice of $\Omega$, we have that $b(\mathbf{x})$ satisfies the Hypothesis ($\star$) with $\mathfrak{B}_0$ by Proposition \ref{prop omega and g}.
We then choose $Q$ to satisfy $0 < Q < \Omega$ and
\begin{equation}
\label{Q bound 1}
Q \cdot \frac{ 2^{d-1}}{g_d(f)} < 1.
\end{equation}
In particular, we may choose $Q$ to satisfy $Q>4$.
We fix these values of $\Omega$ and $Q$ throughout this section.
We note that with these choices of $\Omega$ and $Q$, we have
\begin{equation}
\label{omega bound 2}
0 \ < \ \Omega \  \leq \frac{ (A_d - d Q) \cdot  (\log 2)^{d}}{ 2^{d-1} (d-1) d! }.
\end{equation}
The work of this section is based on \cite{CM} and it is similar to their treatment of
the major arcs. However, we had to tailor their argument to be in terms of the $h$-invariant instead of the Birch rank.

We define the following sums
\begin{equation}
\label{defn Stilde}
\widetilde{S}_{{m}, q } = \sum_{\mathbf{k} \in \mathbb{U}_q^n} e( {b}(\mathbf{k}) \cdot {m}/q ), \ \  B( q ) = \sum_{{m} \in \mathbb{U}_q} \frac{1}{\phi(q)^n} \ \widetilde{S}_{ {m}, q },
\ \ \text{ and } \ \  \mathfrak{S}( N) = \sum_{q \leq (\log N)^C}  B( q ),
\end{equation}
where $\phi$ is Euler's totient function.
Recall we denote $\mathfrak{B}_0 = [0,1]^n$.
We have the following estimate on the major arcs which is a consequence of \cite[(6.1)]{CM} and \cite[Lemma 6]{CM},
and leave the details to the reader.
We remark that although it is assumed in \cite[Lemma 6]{CM} that $C$ is sufficiently large,
it in fact follows from their proof that assuming $C > 0$ is sufficient.
\begin{lem}[Lemma 6, \cite{CM}]
\label{Lemma 6 in CM} Let $c>0$, $C>0$, $q \leq (\log N)^C$, and $m \in \mathbb{U}_q$.
Then we have
$$
\int_{\mathfrak{M}_{{m}, q}(C)} T({b}; {\alpha} ) \ {d}{\alpha}
=
\frac{1}{\phi(q)^n} \ \widetilde{S}_{{m}, q } \  J_0 + O\left( \frac{N^{n - d}}{(\log N)^{c}} \right),
$$
where
$$
J_0 = \int_{|\tau| \leq N^{-d} (\log N)^C } \int_{\mathbf{u} \in N \mathfrak{B}_0 } e( \tau b(\mathbf{u}) ) \ \mathbf{d} \mathbf{u} \ d \tau.
$$
\end{lem}
Note $J_0$ is independent of ${m}$ and  $q$.
We now simplify the expression for $J_0$.
Let
$$
\mathcal{I}( {\eta}) = \int_{ \mathfrak{B}_0 } e( \eta f( \boldsymbol{\xi}) ) \ \mathbf{d} \boldsymbol{\xi}
$$
For any $\varepsilon > 0$, the inner integral of $J_0$ can be expressed as
\begin{eqnarray}
\int_{\mathbf{u} \in N \mathfrak{B}_0 } e( \tau b(\mathbf{u})  ) \ \mathbf{d} \mathbf{u}
&=& \int_{\mathbf{u} \in N \mathfrak{B}_0 } e( \tau f(\mathbf{u}) ) \ \mathbf{d} \mathbf{u} + O(N^{n - 1 + \varepsilon})
\notag
\\
&=& N^n \int_{\boldsymbol{\xi} \in \mathfrak{B}_0 } e( N^d \tau f(\boldsymbol{\xi})   ) \ \mathbf{d} \boldsymbol{\xi} + O(N^{n - 1 + \varepsilon})
\notag
\\
&=& N^n \cdot \mathcal{I}( N^d \tau ) + O(N^{n - 1 + \varepsilon}),
\notag
\end{eqnarray}
where we used the change of variable $\mathbf{u} =  N \boldsymbol{\xi}$ to obtain the second equality above.

We define
$$
J(L) = \int_{|\eta| \leq L } \mathcal{I}(\eta) \ d \eta.
$$
Then we can simplify $J_0$ as
$$
J_0 = N^{n-d} \cdot J( (\log N)^C ) + O(N^{n-d-1 + \varepsilon} (\log N)^C).
$$
Since we have $\Omega > 2$ and the Hypothesis $(\star)$, and in particular the restricted Hypothesis $(\star)$,
it follows by \cite[Lemma 8.1]{S} that
\begin{equation}
\label{(3.9) is S}
\mathcal{I}( {\eta}) \ll \min (1 , |{\eta}|^{-2} ).
\end{equation}
As stated in \cite[Section 3]{S}, it follows from ~(\ref{(3.9) is S}) that
$$
\mu(\infty) = \int_{\mathbb{R}} \mathcal{I}( {\eta}) \ d {\eta}
$$
exists.
Furthermore, we have
\begin{equation}
\label{(3.9') is S}
\Big{|} \mu(\infty) - J(L) \Big{|} \ll L^{- 1}.
\end{equation}
We also have $\mu(\infty) > 0$ if the equation $f(\mathbf{x}) = 0$ has a
non-singular real solution in the interior of $\mathfrak{B}_0 = [0,1]^n$ (see \cite[pp. 704]{CM}).

Therefore, we obtain the following estimate as a consequence of the definition of the major arcs and Lemma \ref{Lemma 6 in CM}.
\begin{lem}
\label{lemma major arc estimate}
Suppose $h(f) > A_d$, where we define $A_d$ as in ~(\ref{def of Ad}). Then given any $c>0$, there exists $C>0$  such that we have
$$
\int_{\mathfrak{M}(C) }T({b}; {\alpha} ) \ {d} {\alpha}
=
\mathfrak{S}(N) \mu(\infty) N^{n - d} + O\left( \mathfrak{S}(N) \frac{ N^{n - d}}{(\log N)^C} +  \frac{N^{n - d}}{(\log N)^c} \right).
$$
\end{lem}

\subsection{Singular Series}
\label{section singular series}
We obtain the following estimate on the exponential sum $\widetilde{S}_{{m}, q }$ defined in ~(\ref{defn Stilde}).
\begin{lem}
\label{to bound local factor}
Suppose $h(f) > A_d$, where we define $A_d$ as in ~(\ref{def of Ad}).
Let $p$ be a prime and let $q = p^t$, $t \in \mathbb{N}.$ For $m \in \mathbb{U}_q$, we have the following bounds
\begin{eqnarray}
\notag
\widetilde{S}_{ {m}, q} \ll
\left\{
    \begin{array}{ll}
         q^{n-Q},
         &\mbox{if } t \leq d ,\\
         p^Q q^{n-Q},
         &\mbox{if } t > d,
    \end{array}
\right.
\end{eqnarray}
where the implicit constants are independent of $p$.
\end{lem}

\begin{proof}
We consider the two cases $t \leq d$ and $t > d$ separately.
We apply the inclusion-exclusion principle to bound $\widetilde{S}_{ {m}, q }$ when $q = p^t$ and $t \leq d$,
\begin{eqnarray}
\label{bound on S tilde}
\widetilde{S}_{ {m}, q }
&=& \sum_{\mathbf{k} \in (\mathbb{Z} / q \mathbb{Z})^n} \prod_{i=1}^n \left( 1 -  \sum_{u_i \in \mathbb{Z} / p^{t-1} \mathbb{Z} } \mathbf{1}_{k_i = p u_i}  \right)
e( {b}(\mathbf{k}) \cdot {m}/q )
\notag
\\
\notag
&=& \sum_{ I \subseteq \{ 1, 2, \ldots, n \}} (-1)^{|I|}
\sum_{ \mathbf{u} \in  (\mathbb{Z} / p^{t-1} \mathbb{Z} )^{|I|} } \
\sum_{\mathbf{k} \in (\mathbb{Z} / q \mathbb{Z})^n}
F_I(\mathbf{k}; \mathbf{u}) e( {b}(\mathbf{k}) \cdot {m}/q ),
\notag
\end{eqnarray}
where
$\mathbf{1}_{k_i = p u_i}$ denotes a characteristic function and
$$
F_I(\mathbf{k}; \mathbf{u}) = \prod_{i \in I} \mathbf{1}_{k_i = p u_i}
$$
for $\mathbf{u} \in (\mathbb{Z} / p^{t-1} \mathbb{Z})^{|I|}$.
In other words, $F_I(\mathbf{k}; \mathbf{u})$ is the characteristic function of the set $H_{I, \mathbf{u}} = \{ \mathbf{k} \in (\mathbb{Z} / q \mathbb{Z})^{n} : k_i = p u_i \ (i \in I)\}$.
We now bound the summand in the final expression of ~(\ref{bound on S tilde}) by further considering two cases, $|I| \geq t Q$ and $|I| < t Q$.
In the first case $|I| \geq t Q$, we use the following trivial estimate
\begin{eqnarray}
\Big{|} \sum_{ \mathbf{u} \in  (\mathbb{Z} / p^{t-1} \mathbb{Z})^{|I|} } \
\sum_{\mathbf{k} \in (\mathbb{Z} / q \mathbb{Z})^n}
F_I(\mathbf{k}; \mathbf{u}) e( {b}(\mathbf{k}) \cdot {m}/q ) \Big{|}
\leq
p^{(t-1) |I|} (p^t)^{n - |I|}
\notag
= q^{n - |I|/t}
\notag
\leq q^{n - Q}.
\notag
\end{eqnarray}

On the other hand, suppose  $|I| < t Q$.
Let $\mathfrak{g}_b(\mathbf{x})$ be the polynomial obtained by
substituting $x_i = p u_i \ (i \in I)$ to ${b}(\mathbf{x})$.
Thus $\mathfrak{g}_b(\mathbf{x})$ is a polynomial in $n - |I|$ variables.
We can also deduce easily that the degree $d$ portion of $\mathfrak{g}_b(\mathbf{x})$, which we denote
$f_{\mathfrak{g}_b}$, is obtained by substituting $x_i = 0 \ (i \in I)$ to the degree $d$ portion of ${b}(\mathbf{x})$.
Hence, we have
$$
f_{\mathfrak{g}_b } =  {f} |_{x_i = 0 \ (i \in I)}.
$$
Consequently, we obtain by Lemma \ref{h ineq 1} that
$$
h(f_{\mathfrak{g}_b }) \geq h( {f}) - |I| > h( {f}) - dQ > A_d - dQ.
$$
By our choice of $Q$ and $\Omega$, and from ~(\ref{h and g}) and ~(\ref{omega bound 2}), we have
$$
0 \ < \ Q \ < \ \Omega \  < \frac{ h(f_{\mathfrak{g}_b}) \cdot (\log 2)^{d}}{ 2^{d-1} (d-1) d! }
\leq \  \frac{g_d( f_{\mathfrak{g}_b} )}{2^{d-1} (d-1) }.
$$
Therefore, with these notations we have by Lemma \ref{bound on E} that
\begin{eqnarray}
\sum_{\mathbf{k} \in (\mathbb{Z} / q \mathbb{Z})^n} F_I(\mathbf{k}; \mathbf{u}) e( {b}(\mathbf{k}) \cdot {m}/q )
= \sum_{\mathbf{s} \in (\mathbb{Z} / q \mathbb{Z})^{n - |I|}} e( \mathfrak{g}_b(\mathbf{s}) \cdot {m}/q )
= q^{n - |I|} E (  \mathfrak{g}_b,  q ;{m}/q )
\ll q^{n - |I| - Q}.
\notag
\end{eqnarray}
Thus, we obtain
$$
\sum_{ \mathbf{u} \in  (\mathbb{Z} / p^{t-1} \mathbb{Z})^{|I|} } \
\sum_{\mathbf{k} \in (\mathbb{Z} / q \mathbb{Z})^n}
F_I(\mathbf{k}; \mathbf{u}) e(  {b}(\mathbf{k}) \cdot  {m}/q )
\ll (p^{t-1})^{|I|} q^{n - |I| - Q} \leq q^{n-Q}.
$$
Consequently, combining the two cases $|I| \geq tQ$ and $|I| < tQ$ together, we obtain
$$
\widetilde{S}_{ {m}, q} \ll q^{n-Q}
$$
when $t \leq d$.

We now consider the case $q = p^t$ when $t>d$. By the definition of $\widetilde{S}_{ {m}, q }$, we have
\begin{eqnarray}
\label{widetile S part 3-1}
\widetilde{S}_{ {m}, q }
= \sum_{\mathbf{k}' \in \mathbb{U}_p^n} \  \sum_{ \mathbf{s} \in (\mathbb{Z} / (p^{t-1} \mathbb{Z}) )^n  } e( {b}(\mathbf{k}' + p \mathbf{s}) \cdot {m}/q )
= \sum_{\mathbf{k}' \in \mathbb{U}_p^n} \  \sum_{ \mathbf{s} \in [0,p^{t-1})^n   } e(  {b}(\mathbf{k}' + p \mathbf{s}) \cdot {m}/q ).
\notag
\end{eqnarray}
For each fixed $\mathbf{k}' \in \mathbb{U}_p^n$, we have
$$
b(\mathbf{k}' + p \mathbf{s}) = p^d  f( \mathbf{s}) + \chi_{p, \mathbf{k}'}(\mathbf{s}),
$$
where $\chi_{p, \mathbf{k}'}(\mathbf{x})$ is a polynomial of degree at most $d-1$ and its coefficients depend on $p$ and $\mathbf{k}'$.
We apply Corollary \ref{cor 15.1 in S} with $r_d=1$, $\psi(\mathbf{x}) = f( \mathbf{x}) + \frac{1}{p^d} \ \chi_{p, \mathbf{k}'}(\mathbf{x})$, $ {\alpha} = {m}/p^{t-d}$, $\mathfrak{B} = [0,1)^{n}$, and $P = p^{t-1}$.
Let $\varepsilon' > 0$ be sufficiently small.
Recall from ~(\ref{Q bound 1}) that our choice of $Q>0$ satisfies $$Q \cdot \frac{2^{d-1}}{g_d( f)} < 1. $$
Let $\gamma_d$ and $\gamma_d'$ be as in the paragraph before Corollary \ref{cor 15.1 in S} with $\mathbf{f} = \{ f \}$ and $r_d = 1$.
Suppose the alternative $(ii)$ of Corollary \ref{cor 15.1 in S} holds. Then we know there exists $n_0 \in \mathbb{N}$ such that
$$
n_0 \ll (p^{t-1}-1)^{Q \gamma_d + \varepsilon'}
$$
and
\begin{equation}
\label{ineq of n0 in sing ser'}
\| n_0 ( {m}/p^{t-d})  \| \ll (p^{t-1}-1)^{-d + Q \gamma_d + \varepsilon'} \leq \left( \frac{1}{2}p^{t-1} \right)^{-d + Q \gamma_d + \varepsilon'}.
\end{equation}
However, this is not possible once $p^t$ is sufficiently large with respect to $n,d, \varepsilon', Q$, and $ {f}$, for the following reason. First note that $n_0$ can not be divisible by $p^{t-d}$ for $p^t$ sufficiently large, because
$Q \gamma_d + \varepsilon' < Q \gamma_d' < 1$. Since $n_0 \in \mathbb{N}$ is not divisible by $p^{t-d}$ and $(m,p)=1$, we have
$$
\| n_0 ( {m}/p^{t-d})  \| \geq \frac{1}{p^{t-d}},
$$
which contradicts ~(\ref{ineq of n0 in sing ser'}) for $p^t$ sufficiently large.
Thus by Corollary \ref{cor 15.1 in S}, we can bound the inner sum of ~(\ref{widetile S part 3-1}) by
\begin{eqnarray}
\sum_{ \mathbf{s} \in [0, p^{t-1})^n  } e \left( \left( f( \mathbf{s}) + \frac{1}{p^d} \ \chi_{p, \mathbf{k}'}(\mathbf{s}) \right) \cdot {m}/ p^{t-d} \right)
&\ll&
\notag (p^{t-1})^{n - Q},
\end{eqnarray}
where the implicit constant depends at most on $n,d, \varepsilon', Q$, and $ {f}$.
Therefore, we can bound ~(\ref{widetile S part 3-1}) as follows
\begin{eqnarray}
\widetilde{S}_{{m}, q }
\leq
\sum_{\mathbf{k}' \in \mathbb{U}_p^n} \Big{|} \sum_{ \mathbf{s} \in [0, p^{t-1})^n  } e \left( \left( f( \mathbf{s}) + \frac{1}{p^d} \ \chi_{p, \mathbf{k}'}(\mathbf{s}) \right) \cdot  {m}/ p^{t-d} \right) \Big{|}
\notag
\ll
p^n (p^{t-1})^{n - Q}
=
p^{Q} q^{n - Q}.
\notag
\end{eqnarray}
\end{proof}

For each prime $p$, we define
\begin{equation}
\label{def mu p}
\mu(p) =  1  + \sum_{t=1}^{\infty} B( p^t),
\end{equation}
which converges absolutely provided that $h(f) > A_d$ as we see in the following lemma.
As stated in \cite{CM}, by following the outline of L. K. Hua \cite[Chapter VII, \S 2, Lemma 8.1]{H} one can show that $B(q)$
is a multiplicative function of $q$. Therefore, we consider the following identity
\begin{equation}
\label{def sigular series}
\mathfrak{S}(\infty) := \lim_{N \rightarrow \infty} \mathfrak{S}(N) = \prod_{p \ \text{prime}} \mu(p).
\end{equation}

\begin{lem}
\label{singular series lemma}
There exists $\delta_1 > 0$ such that for each prime $p$, we have
$$
\mu(p) = 1 + O(p^{-1 - \delta_1}),
$$
where the implicit constant is independent of $p$.
Furthermore, we have
$$
\Big{|} \mathfrak{S}(N) -  \mathfrak{S}(\infty) \Big{|} \ll (\log N)^{-C \delta_2 }
$$
for some $\delta_2 > 0$.
\end{lem}
Therefore, the limit in ~(\ref{def sigular series}) exists, and the product in ~(\ref{def sigular series}) converges.
We leave the details  that these two quantities are equal to the reader.
\begin{proof}
Recall our choice of $Q$ satisfies $Q > 4$. Let $\varepsilon_0 > 0$ be sufficiently small such that
$\widetilde{Q} = Q - \varepsilon_0 > 4 \geq 2d/(d-1)$.
We substitute $Q = \widetilde{Q} + \varepsilon_0$ into the bounds in Lemma \ref{to bound local factor}. It is then clear that we may assume
the implicit constant in Lemma \ref{to bound local factor} is $1$ for $p$ sufficiently large
with the cost of using $\widetilde{Q}$ in place of $Q$.
For any $t \in \mathbb{N}$, we know that $\phi(p^t) = p^t(1 - 1/p) \geq \frac12 p^t$.
Therefore, by considering the two cases as in Lemma \ref{to bound local factor}, we obtain
\begin{eqnarray}
| \mu(p) - 1 |
\ll
\sum_{1 \leq t \leq d} p^{t} p^{-nt} p^{nt - t \widetilde{Q} }
+
\sum_{ t > d} p^{t} p^{-nt} p^{\widetilde{Q} + nt - t \widetilde{Q} }
\notag
\ll
p^{1 - \widetilde{Q} }
+
p^{\widetilde{Q}} p^{-(d+1)(\widetilde{Q}-1)}
\notag
\notag
\ll
p^{-1 - \delta_1},
\notag
\end{eqnarray}
for some $\delta_1 > 0$. We note that the implicit constants in $\ll$ are independent of $p$ here.

Let $q = p_1^{t_1} \cdots p_{v}^{t_{v}}$ be the prime factorization of $q \in \mathbb{N}$.
Without loss of generality, suppose we have $t_j \leq d \ (1 \leq j \leq v_0)$ and
$t_j > d \ (v_0 < j \leq v)$.
By the multiplicativity of $B(q)$, it also follows from Lemma \ref{to bound local factor} that
\begin{eqnarray}
B(q)
=
B(p_1^{t_1})  \cdots   B(p_{v}^{t_{v}})
\notag
\ll
q^{1 - \widetilde{Q}} \cdot  \left(  \prod_{j=v_0 + 1}^{v}  p_j^{ \widetilde{Q} } \right)
\notag
\leq
q^{1 - \widetilde{Q}} \cdot  q^{\widetilde{Q}/d}
\leq
q^{- 1 - \delta_2},
\notag
\end{eqnarray}
for some $\delta_2 > 0$. We note that the implicit constant in $\ll$ is independent of $q$ here, because
the implicit constant in Lemma \ref{to bound local factor} is $1$ for $p$ sufficiently large as mentioned above.
Therefore, we obtain
\begin{eqnarray}
\Big{|} \mathfrak{S}(N) -  \mathfrak{S}(\infty) \Big{|}
\leq
\sum_{q > (\log N)^C } |  B(q) |
\notag
\ll
\sum_{q > (\log N)^C } q^{- 1 - \delta_2}
\ll
(\log N)^{- C \delta_2 }.
\end{eqnarray}
\end{proof}

Let $\nu_t(p)$ denote the number of solutions $\mathbf{x} \in (\mathbb{U}_{p^t})^n$
to the congruence
\begin{eqnarray}
b( \mathbf{x} ) \equiv 0 \  (\text{mod } p^t).
\end{eqnarray}
It can be deduced that
\begin{eqnarray}
1 + \sum_{j=1}^t B(p^j)
\notag
= \frac{1}{\phi(p^t)^n} \sum_{\mathbf{k} \in (\mathbb{U}_{p^t})^{n}}   \sum_{ {m} \in \mathbb{Z}/(p^t \mathbb{Z} ) } e \left(  {b}(\mathbf{k})  \cdot {m}/p^t  \right)
= \frac{p^{t}}{\phi(p^t)^n } \ \nu_t(p).
\notag
\end{eqnarray}
Therefore, provided  $h({b}) > A_{d}$ we obtain
$$
\mu(p) = \lim_{t \rightarrow \infty}  \frac{  p^{t } \ \nu_t(p) }{ \phi(p^t)^n }.
$$
At this point we refer the reader to \cite[pp. 704, 736]{CM} to conclude $\mu(p) > 0$
if the equation $b(\mathbf{x}) = 0$ has a non-singular solution in $\mathbb{Z}_p^{\times}$, the units of $p$-adic integers.
It then follows from Lemma \ref{singular series lemma} that if the equation $b(\mathbf{x}) = 0$ has a non-singular solution in $\mathbb{Z}_p^{\times}$
for every prime $p$, then $\prod_{p \ \text{prime}}\mu(p) > 0.$
Finally, we let $C_b = \mu(\infty) \prod_{p \ \text{prime}}\mu(p)$ and 
Theorem \ref{the main theorem} follows as a  consequence of Lemmas \ref{lemma major arc estimate} and \ref{singular series lemma}, and
Proposition \ref{prop minor arc bound}.

\end{document}